\documentclass[letterpaper,11pt]{article}

\usepackage[utf8]{inputenc}
\usepackage[english]{babel}

\usepackage{latexsym}
\usepackage{amssymb}
\usepackage{amsthm}
\usepackage{mathrsfs}
\usepackage{amstext}
\usepackage{graphicx}  
\usepackage{amsfonts}
\usepackage{wrapfig}
\usepackage{sidecap} 
\usepackage{latexsym}
\usepackage{pdfpages}
\usepackage{epic}
\usepackage{fullpage}
\usepackage{color}
\usepackage{curves}
\usepackage{subfigure}
\usepackage{setspace}
\usepackage{amsmath}
\numberwithin{equation}{section}
\usepackage{anysize}
\usepackage{rotating}
\usepackage{enumerate} 
\usepackage{hyperref}
\usepackage{bbm}
\usepackage{graphicx}
\usepackage{subfigure}

\newtheorem{theorem}{Theorem}[section] 
\newtheorem{definition}[theorem]{Definition}

\newtheorem{proposition}[theorem]{Proposition}
\newtheorem{lemma}[theorem]{Lemma}
\newtheorem{remark}[theorem]{Remark}
\def\R{{\mathbb R}}

\renewcommand{\leq}{\leqslant}
\renewcommand{\geq}{\geqslant}
\numberwithin{equation}{section}

\newcommand{\pt}{\partial_t}

\newcommand{\pnu}{\partial_\nu}

\title{Local null controllability of a cubic Ginzburg-Landau equation with dynamic boundary conditions}
\author{
	Nicol\'as Carre\~{n}o\thanks{Departamento de Matem\'atica, Universidad T\'ecnica Federico Santa Mar\'{\i}a, Casilla 110-V, Valpara\'{\i}so, Chile e-mail: nicolas.carrenog@usm.cl.}
	\and 
	Alberto Mercado\thanks{Departamento de Matem\'atica, Universidad T\'ecnica Federico Santa Mar\'{\i}a, Casilla 110-V, Valpara\'{\i}so, Chile e-mail: alberto.mercado@usm.cl.
	}
	\and 
	Roberto Morales\thanks{Departamento de Matem\'atica, Universidad T\'ecnica Federico Santa Mar\'{\i}a, Casilla 110-V, Valpara\'{\i}so, Chile e-mail: roberto.moralesp@usm.cl.}
	}
	
\begin{document}
\maketitle

\begin{abstract}
This paper deals with    controllability properties of a cubic Ginzburg-Landau equation with dynamic boundary conditions. More precisely, we prove a local null controllability result by using a single control supported in a small subset of the domain. In order to achieve this result, we firstly linearize the system around the origin and we analyze it by the duality approach and an appropriate Carleman estimate. Then, by using an inverse function theorem, the local null controllability of the nonlinear system is proven. 
\end{abstract}


\section{Introduction and main results} 
\subsection{Introduction} 
Let $\Omega\subset \mathbb{R}^d$ ($d\geq 2$) be a bounded domain with boundary $\Gamma:=\partial \Omega$ of class $C^2$.
Given the parameters $a,b,c>0$,  $\alpha, \gamma \in \mathbb{R}\setminus \{0\}$, we consider the following cubic Ginzburg-Landau equation with dynamic boundary conditions 
\begin{align}
\label{intro:eq:01}
\begin{cases}
\pt u -a(1+\alpha i)\Delta u + c(1+\gamma i)|u|^2u =\mathbbm{1}_\omega h,&\text{ in }\Omega\times (0,T),\\
\pt u_\Gamma + a(1+\alpha i)\pnu u -b(1 + \alpha i)\Delta_\Gamma u_\Gamma +c(1+\gamma i)|u_\Gamma|^2 u_\Gamma =0,&\text{ on }\Gamma\times (0,T),\\
u=u_\Gamma,&\text{ on }\Gamma\times (0,T),\\
(u(0),u_\Gamma(0))=(u_0,u_{\Gamma,0}),&\text{ in }\Omega\times \Gamma. 
\end{cases}
\end{align}

Here, $(u,u_\Gamma)$ is the state of the system, $(u_0,u_{\Gamma,0})$ the initial conditions and $h\in L^2(\omega\times (0,T);\mathbb{C})$ is a control acting on $\omega\subset \Omega$. We denote by $\Delta_\Gamma$ the Laplace-Beltrami operator on $\Gamma$ and by $\pnu y$ the normal derivative associated to the outward normal $\nu$ of $\Omega$.

We notice that \eqref{intro:eq:01} can be seen as a coupled system in the variables $(u,u_\Gamma)$, which is controlled by a single control $h$ in a (small) subset $\omega$ of $\Omega$. This means that the first equation of \eqref{intro:eq:01} is controlled directly by the action of the control, while the second equation is being controlled through the side condition $u=u_\Gamma$ on $\Gamma\times (0,T)$.  

The main objective of this work is to obtain the local null controllability of system \eqref{intro:eq:01} in $X$ (where $X$ is an appropriate Banach space), i.e., we will prove the existence of a number $\delta>0$ such that, for every initial state $(u_0,u_{\Gamma,0})\in X$ which fulfills
\begin{align*}
	\|(u_0,u_{\Gamma,0})\|_{X}\leq \delta,
\end{align*}
we can find a control $h\in L^2(\omega\times (0,T))$ such that the associated solution $(u,u_\Gamma)$ of \eqref{intro:eq:01} satisfies 
\begin{align*}
	y(\cdot, T)=0,\text{ in }\Omega,\quad y_\Gamma(\cdot,T)=0,\text{ on }\Gamma.
\end{align*}

\subsection{Previous results}
The cubic complex Ginzburg-Landau equation is one of the most studied nonlinear equations used to model  physical phenomena. 
This  equation  has been used to describe 
several phenomena ranging from nonlinear waves, second-order phase transitions, superconductivity, superfluidity and Bose-Einstein condensation to liquid crystals and strings in field theory. For a detailed description of  relevant  applications  in  different fields, see \cite{aranson2002world} and \cite{garcia2012complex}.

The existence and uniqueness of solutions of nonlinear Ginzburg-Landau equations with Dirichlet or periodic boundary conditions have been intensely investigated in several papers. For instance, we refer to \cite{Levermore1996}, \cite{Gao2004}, \cite{Gao2007}, \cite{Fu2006}, \cite{Chen1998}, \cite{Bu1990},  \cite{Bartuccelli1990} and the references therein. Concerning controllability properties of the Ginzburg-Landau equation with Dirichlet boundary conditions, only a few papers have been devoted to the study of the controllability  of such problems. In \cite{Aamo2005}, the stabilization of the linearized Ginzburg-Landau model with Dirichlet boundary conditions around an unstable equilibrium state is studied. Moreover, in \cite{Fu2006}, the author develop a Carleman inequality for an operator of the form 
\begin{align*}
	(a+ib)\pt +\text{div}(A\cdot \nabla),
\end{align*} 
with $A$ being a smooth, uniformly elliptic matrix, and a null controllability result for the linear PDE with a distributed control. In 2009, L. Rosier and B.-Y. Zhang in \cite{Rosier2009} proved a controllability result for the nonlinear case. In this case, the control acts on a part of the boundary and the proof is based on a suitable Carleman estimate for the linear adjoint system. Then, combining a fixed-point argument together with the theory of sectorial operators, the authors obtained a local controllability result for a wide class of nonlinearities. In particular, controllability results for the cubic and quintic Complex Ginzburg-Landau are provided.   

Recently, some results on inverse problems and controllability issues have been obtained for PDEs with dynamic boundary conditions, see for instance \cite{Maniar2017}, \cite{Khoutaibi2022}, \cite{BenHassi2021}, \cite{BenHassi2022inverse}, \cite{BenHassi2022},  and \cite{lecaros2022discrete} for the heat equation, \cite{Gal2017} for the wave equation, and \cite{mercado2022exact} for the Schr\"odinger operator. In these works, the authors has been used the duality equivalence to prove the associated observability inequality by using Carleman estimates. At this level, we point out that it is not evident at all that such systems can be controlled by the action of a single control due to the tangential derivative terms. In fact, in the case of the linear wave equation with mixed boundary conditions (oscilatory boundary conditions and Dirichlet boundary conditions) \cite{Gal2017}, the authors obtain exact controllability results where the control region is on the whole boundary (and therefore on the whole system). 
On the other hand, in  a similar setting, in \cite{mercado2022exact} the authors obtained the exact controllability of the linear Schr\"odinger equation with dynamic boundary conditions. In this case, the control acts only in a part of the boundary. To prove the associated observability inequality, the authors used a Global Carleman estimate for the Schr\"odinger operator, where the weight function is adapted to the geometric properties of the domain. 


Concerning the Ginzburg-Landau equation with dynamic boundary conditions, we mention \cite{Correa2018}, where well-posedness of linear/nonlinear of such models is obtained and long time behavior of solutions is characterized when Lipschitz nonlinearities are considered. However, to the best of the authors' knowledge, this is the first time that the null controllability for the cubic Ginzburg-Landau is studied. 

\subsection{General setting}
In this section, we set up the notation and terminology used in this paper. The set $\Gamma=\partial \Omega$ can be seen as an $(d-1)$-dimensional compact Riemannian submanifold equipped by the Riemmanian metric $g$  induced by the natural embedding $\Gamma \subset \mathbb{R}^d$. In addition, we shall denote by $dS$ the $(d-1)$-Lebesgue measure for $\Gamma$.

Since we are considering dynamic boundary conditions, we need to define some differential operators on $\Gamma$, 
which  
 can be defined in terms of the associated metric. However, for our purposes, it will be enough to use the most important properties of the underlaying operators and spaces. The details can be found, for instance, in \cite{Taylor}. For the sake of completeness, we recall some of those properties.

The tangential gradient $\nabla_\Gamma$ of $y_\Gamma$ at each point $x\in \Gamma$ can be seen as the projection of the standard Euclidean gradient $\nabla y$ onto the tangent space of $\Gamma$ at $x\in \Gamma$, where $y_\Gamma$ is the trace of $y$ on $\Gamma$, i.e., we have the following equation
\begin{align*}
\nabla_\Gamma y_\Gamma=\nabla y - \nu  \pnu y,
\end{align*}
where $y=y_\Gamma$ on $\Gamma$ and $\pnu y$ is the normal derivative associated to the outward normal $\nu$. In this way, the tangential divergence $\text{div}_\Gamma$ in $\Gamma$ is defined by 
\begin{align*}
\text{div}_\Gamma(F_\Gamma):H^1(\Gamma;\mathbb{R})\to \mathbb{R},\quad y_\Gamma \mapsto -\int_\Gamma F_\Gamma \cdot \nabla_\Gamma y_\Gamma dS.
\end{align*}

The Laplace-Beltrami operator is given by $\Delta_\Gamma y_\Gamma: =\text{div}(\nabla_\Gamma y_\Gamma)$, for all $y_\Gamma \in H^2(\Gamma;\mathbb{R})$. In particular, the surface divergence theorem holds:
\begin{align*}
\int_\Gamma \Delta_\Gamma y_\Gamma z_\Gamma dS =-\int_\Gamma \nabla_\Gamma y_\Gamma \cdot \nabla_\Gamma z_\Gamma dS,\quad \forall y_\Gamma\in H^2(\Gamma;\mathbb{R}),\quad \forall z_\Gamma \in H^1(\Gamma;\mathbb{R}).
\end{align*}


In order to simplify the notation, here and subsequently, the function spaces refer to complex-valued functions unless otherwise stated.

For $1\leq p\leq +\infty$, we consider the Banach space $\mathbb{L}^p:=L^p(\Omega)\times L^p(\Gamma)$, endowed by the norm given by the relation
\begin{align*}
	\|(u,u_\Gamma)\|_{\mathbb{L}^p}^2:=\|u\|_{L^p(\Omega)}^2 + \|u_\Gamma\|_{L^p(\Gamma)}^2.
\end{align*}

In particular, for $p=2$, the space $\mathbb{L}^2:=L^2(\Omega)\times L^2(\Gamma)$ is a (real) Hilbert space equipped with the scalar product
\begin{align*}
\langle (u,u_\Gamma),(v,v_\Gamma) \rangle_{\mathbb{L}^2}:=\Re\int_\Omega u\overline{v} dx+ \Re\int_\Gamma u_\Gamma \overline{v_\Gamma}dS. 
\end{align*}

For $k\in \mathbb{N}$, we also introduce the space
\begin{align*}
\mathbb{H}^k :=\{(y,y_\Gamma)\in H^k(\Omega)\times H^k(\Gamma)\,;\, y\big|_\Gamma =y_\Gamma\},
\end{align*}
where $H^k(\Omega)$ and $H^k(\Gamma)$ are the usual Sobolev spaces.
\subsection{Main result}  

Our main result states the local null controllability on the space $\mathbb{H}^1$:

\begin{theorem}{}
	\label{main:thm}
	Suppose that $d=2$ or $d=3$. Let $a,b,c>0$, $\alpha, \gamma \in \mathbb{R}\setminus \{0\}$. Then, for every $T>0$ and $\omega \Subset \Omega$, there exists $\delta>0$ such that, for every $(y_0,y_{\Gamma,0})\in \mathbb{H}^1$ satisfying
	\begin{align*}
		\|(u_0,u_{\Gamma,0})\|_{\mathbb{H}^1}\leq \delta,
	\end{align*}
	there exists  a control $h\in L^2(\omega\times (0,T))$
	such that the unique 
corresponding  solution $(u,u_\Gamma)$ of \eqref{intro:eq:01} 
satisfies 
	\begin{align*}
		u(\cdot,T)=0,\text{ in }\Omega,\quad u_\Gamma (\cdot,T)=0,\text{ on }\Gamma,
	\end{align*}
\end{theorem}

To prove Theorem \ref{main:thm} we first deduce a null controllability result for a linear system associated to \eqref{intro:eq:01}:
	\begin{align}
		\label{linear:pb:intro}
		\begin{cases}
			\pt y-a(1+\alpha i)\Delta y=f+\mathbbm{1}_\omega h,&\text{ in }\Omega\times (0,T),\\
			\pt y_\Gamma +a(1+\alpha i)\pnu y-b(1+\alpha i)\Delta_\Gamma y_\Gamma =f_\Gamma,&\text{ on }\Gamma\times (0,T),\\
			y=y_\Gamma,&\text{ on }\Gamma \times (0,T),\\
			(y,y_\Gamma)(0)=(y_0,y_{\Gamma,0}),&\text{ in }\Omega\times \Gamma,
		\end{cases}
	\end{align} 
	where $(f,f_\Gamma)$ will be taken to decrease exponentially to zero in $t=T$. Then, we prove a new Carleman estimate for the adjoint system of \eqref{linear:pb:intro} (see estimate \eqref{linear:system:z} below). This will provide existence of a unique solution to a suitable variational problem, from which we define a solution $(y,y_\Gamma,h)$ to \eqref{linear:pb:intro} such that $y(T)=0$ in $\overline{\Omega}$. Moreover, the solution is such that $e^{C/(T-t)}(y,y_\Gamma,h)\in \mathbb{L}^2 \times L^2(\omega\times (0,T))$, for some constant $C>0$.
	Finally, by an inverse mapping theorem, we deduce the null controllability for the nonlinear system.
	
	The rest of the paper is organized as follows. In Section \ref{section:Existence}, we establish the existence and uniqueness of solutions of \eqref{intro:eq:01} and \eqref{linear:pb:intro}. In Section \ref{Section:Carleman}, we prove a suitable Carleman estimate for the Ginzburg-Landau operator with dynamic boundary conditions. In Section \ref{section:Observability}, we prove the observability estimate for the adjoint system and prove the null controllability of \eqref{linear:pb:intro}. Finally, in Section \ref{section:proof:main:result} we prove the Theorem \ref{main:thm}. 
	
\section{Existence and uniqueness of solutions}
\label{section:Existence}
In this section, we present new results concerning existence and uniqueness for the Ginzburg-Landau equations with dynamic boundary conditions. 

\subsection{Linear problem}
We consider the Cauchy problem
\begin{align}
	\label{linearized:problem:01}
	\begin{cases}
		Lu=f,&\text{ in }\Omega\times (0,T),\\
		L_\Gamma(u,u_\Gamma) =f_\Gamma,&\text{ on }\Gamma\times (0,T),\\
		u=u_\Gamma,&\text{ on }\Gamma\times (0,T),\\
		(u(0),u_\Gamma (0))=(u_0,u_{\Gamma,0}),&\text{ in }\Omega\times \Gamma.
	\end{cases}	
\end{align}
where 
\begin{align}
	\label{def:L:N}
	Lu:=\pt u -a(1+\alpha i)\Delta u,\quad L_\Gamma(u,u_\Gamma):=\pt u_\Gamma +a(1+\alpha i)\pnu u -b(1+\alpha i)\Delta_\Gamma u_\Gamma,
\end{align}
respectively. Notice that the problem \eqref{linearized:problem:01} can be seen in the abstract form
\begin{align}
	\label{abstract:form}
	\begin{cases}
		U'(t)=\mathcal{A}_{GL}U(t)+\mathcal{F}(t),\quad t\in (0,T),\\
		U(0)=U_0,
	\end{cases}
\end{align}
where $\mathcal{A}_{GL}:D(\mathcal{A}_{GL})\subset \mathbb{L}^2 \to \mathbb{L}^2$ is the operator defined by
\begin{align}
	\label{def:AGL}
	\mathcal{A}_{GL}(U):=\left[ 
	\begin{array}{c}
	a(1+\alpha i)\Delta u\\
	-a(1+\alpha i)\pnu u +b(1+\alpha i)\Delta_\Gamma u_\Gamma
	\end{array}	
	\right],\quad \forall\, U:= \left[ 
	\begin{array}{c}
		u\\u_\Gamma
	\end{array}
	\right]\in D(\mathcal{A}_{GL}),   
\end{align}
with domain 
\begin{align*}
	D(\mathcal{A}_{GL}):=\left\{ U=
	\left[\begin{array}{c} 
	u\\u_\Gamma
	\end{array}
	\right]\in \mathbb{H}^1 \,:\, (\Delta u,\Delta_\Gamma u_\Gamma)\in \mathbb{L}^2\right\}
	=\mathbb{H}^2,
\end{align*}
where the last equivalence is justified in \cite{Coclite2009Role} (see also \cite{Gal2015Role}). If $A_W$ is the Wentzell-Laplacian operator introduced in \cite{Maniar2017}, then it is easy to see that $\mathcal{A}_{GL}$ can be written as
\begin{align*}
	\mathcal{A}_{GL}=(1+\alpha i)A_W, \quad D(A_W)=D(\mathcal{A}_{GL})=\mathbb{H}^2.
\end{align*}

Then, arguing as in \cite[Section 2]{Rosier2009} we have the following result:
\begin{proposition}
\label{prop:AGL}
The operator $\mathcal{A}_{GL}$ defined in \eqref{def:AGL} is densely defined and generates an analytic semigroup $(e^{t \mathcal{A}_{GL}})_{t\geq 0}$ in $\mathbb{L}^2$. 
\end{proposition}

According to Proposition \ref{prop:AGL}, the existence and uniqueness of strong solutions of \eqref{abstract:form} in the usual sense are guaranteed. In the next subsection, we provide existence and uniqueness of solutions in appropriate spaces  by energy estimates and density arguments.

\begin{proposition}
	\label{proposition:estimate:L2}
	Suppose that $(u_0,u_{\Gamma,0})\in \mathbb{L}^2$ and $(f,f_\Gamma)\in L^2(0,T;\mathbb{L}^2)$. Then, the weak solution $(u,u_\Gamma)$ of \eqref{linearized:problem:01} belongs to $C^0([0,T];\mathbb{L}^2)\cap L^2(0,T;\mathbb{H}^1)$. Moreover, there exists a constant $C_1>0$ such that the associated solution $(u,u_\Gamma)$ of \eqref{linearized:problem:01} satisfies
	\begin{align}
		\label{prop:weak:solutions:01}
		\|(u,u_\Gamma)\|_{C^0([0,T];\mathbb{L}^2)} + \|(u,u_\Gamma)\|_{L^2(0,T;\mathbb{H}^1)} \leq C_1 \|(f,f_\Gamma)\|_{L^2(0,T;\mathbb{L}^2)} + C_1\|(u_0,u_{\Gamma,0})\|_{\mathbb{L}^2}.
	\end{align}
\end{proposition}
\begin{proof}
	Firstly, we multiply by $\overline{u}$ the first equation of \eqref{linearized:problem:01} and we integrate in $\Omega$. Secondly, we multiply the second equation of \eqref{linearized:problem:01} by $\overline{u}_\Gamma$ and integrate on $\Gamma$. Next, we add these identities and take the real part on the obtained equation. After integration by parts, this yields
	\begin{align*}
		\begin{split} 
		&\dfrac{1}{2} \dfrac{d}{dt} \left(\int_\Omega |u(t)|^2 dx + \int_\Gamma |u_\Gamma(t)|^2dS \right) +a \int_\Omega |\nabla u(t)|^2 dx + b\int_\Gamma |\nabla_\Gamma u_\Gamma (t)|^2 dS\\
		=&\Re\int_\Omega f(t)\overline{u}(t)dx + \Re\int_\Gamma f_\Gamma (t)\overline{u}_\Gamma (t)dS. 
		\end{split}
	\end{align*} 
	
	By Young's inequality, it is easy to check that 
	\begin{align*}
		\|(u,u_\Gamma)\|_{C^0([0,T];\mathbb{L}^2)}^2 + \|(u,u_\Gamma)\|_{L^2(0,T;\mathbb{H}^1)}^2 \leq C \|(f,f_\Gamma)\|_{L^2(0,T;\mathbb{L}^2)}^2 + C\|(u_0,u_{\Gamma,0})\|_{\mathbb{L}^2}^2,
	\end{align*}
	which clearly implies \eqref{prop:weak:solutions:01}.
\end{proof}

\begin{proposition}
	\label{prop:estimate:H2}
	Let $(u,u_{\Gamma,0})\in \mathbb{H}^1$ and $(f,f_\Gamma)\in L^2(0,T;\mathbb{L}^2)$. Then, the associated weak solution $(u,u_\Gamma)$ of \eqref{linearized:problem:01} belongs to $H^1(0,T;\mathbb{L}^2)\cap C^0([0,T];\mathbb{H}^1)\cap L^2(0,T;\mathbb{H}^2)$. Moreover, there exists a constant $C_2>0$ such that $(u,u_\Gamma)$ satisfies
	\begin{align}
	\label{estimate:H2}
		\begin{split}
			&\|(u,u_\Gamma)\|_{H^1(0,T;\mathbb{L}^2)} + \|(u,u_\Gamma)\|_{C^0([0,T];\mathbb{H}^1)}+\|(u,u_\Gamma)\|_{L^2(0,T;\mathbb{H}^2)}\\
			\leq &C_2\|(f,f_\Gamma)\|_{L^2(0,T;\mathbb{L}^2)} + C_2\|(u_0,u_{\Gamma,0})\|_{\mathbb{H}^1}.
		\end{split}
	\end{align}
\end{proposition}

\begin{proof}
	The proof is divided into three steps.
	
	\noindent $\bullet$  Step 1: Our first task is to obtain an $L^2(L^2)$ estimates for $(\pt u,\pt u_\Gamma)$, respectively. In order to do that, we multiply the first equation of \eqref{linearized:problem:01} by $(1-\alpha i) \pt \overline{u}$ and integrate in $\Omega$. In addition, we multiply the second equation of \eqref{linearized:problem:01} by $(1-\alpha i)\pt \overline{u}_\Gamma$ and integrate on $\Gamma$. Then, we sum up these identities and take the real part. This yields
	\begin{align*}
		&\int_\Omega |\pt u(t)|^2 dx + \int_\Gamma |\pt u_\Gamma(t)|^2 dS + \dfrac{1}{2}a(1+\alpha^2) \dfrac{d}{dt} \int_\Omega |\nabla u|^2 dx + \dfrac{1}{2} b(1+\alpha^2) \dfrac{d}{dt}\int_\Gamma |\nabla_\Gamma u_\Gamma|^2 dS\\
		=& \Re \int_\Omega (1-\alpha i)f\pt \overline{u} dx + \Re\int_\Gamma (1-\alpha i)f_\Gamma \pt \overline{u}_\Gamma dS. 
	\end{align*} 
	
	Integrating in time the above inequality, we easily get 
	\begin{align}
		\label{eq:estimate:L2:ptu01}
		\|(u,u_\Gamma)\|_{H^1(0,T;\mathbb{L}^2)}^2 + \|(u,u_\Gamma)\|_{C^0([0,T];\mathbb{H}^1)}^2 \leq C\|(f,f_\Gamma)\|_{L^2(0,T;\mathbb{L}^2)}^2 + C\|(u_0,u_{\Gamma,0})\|_{\mathbb{H}^1}^2.
	\end{align}
	\noindent $\bullet$ Step 2: We shall derive $L^2(H^2)$ estimates for $(u,u_\Gamma)$. To do this, we firstly point out that the estimate of $\pt u$ implies that
	\begin{align*}
		\|\Delta u\|_{L^2(0,T;L^2(\Omega))}\leq C\|(f,f_\Gamma)\|_{L^2(0,T;\mathbb{L}^2)} + C\|(u_0,u_{\Gamma,0})\|_{\mathbb{H}^1}.
	\end{align*}

	By elliptic regularity applied to the first equation of \eqref{linearized:problem:01}, we have
	\begin{align*}
		\|u(t)\|_{H^2(\Omega)} \leq C\|f(t)\|_{L^2(\Omega)} + C\|\pt u(t)\|_{L^2(\Omega)} + C\|u_\Gamma (t)\|_{H^{3/2}(\Gamma)},
	\end{align*}
	a.e. in $(0,T)$. Integrating on $t\in [0,T]$, we deduce that 
	\begin{align}
		\label{eq:estimate:H2:Omega}
		\|u\|_{L^2(0,T;H^2(\Omega))}\leq C_* \|(f,f_\Gamma)\|_{L^2(0,T;\mathbb{L}^2)} + C_* \|(u_0,u_{\Gamma,0})\|_{\mathbb{H}^1} + C_* \|u_\Gamma\|_{L^2(0,T;H^{3/2}(\Gamma))},
	\end{align}
	for some constant $C_*>0$. Now, from the second equation of \eqref{linearized:problem:01}, we deduce that 
	\begin{align}
	\nonumber 
		\|u_\Gamma\|_{L^2(0,T;H^2(\Gamma))}\leq &C\|f_\Gamma\|_{L^2(0,T;L^2(\Gamma))} + C\|\pnu u\|_{L^2(0,T;L^2(\Gamma))} + C\|\pt u_\Gamma\|_{L^2(0,T;L^2(\Gamma))}\\
		\label{eq:estimate:H2:Gamma}
		\leq & C \|(f,f_\Gamma)\|_{L^2(0,T;\mathbb{L}^2)} + C\|(u_0,u_{\Gamma,0})\|_{\mathbb{H}^1} + C\|\pnu u\|_{L^2(0,T;L^2(\Gamma))}.
	\end{align}
	
	Moreover, by interpolation inequalities, we have that for every $0<s<1/2$ and $\varepsilon>0$, there are positive constants $C_s$ and $C_\varepsilon$ such that
	\begin{align}
		\label{eq:estimate:pnu}
		\|\pnu u\|_{L^2(0,T;L^2(\Gamma))}\leq C_s \|u\|_{L^2(0,T;H^{3/2+s}(\Omega))}\leq C_\varepsilon \|u\|_{L^2(0,T;H^1(\Omega))}+ \varepsilon \|u\|_{L^2(0,T;H^2(\Omega))}.
	\end{align}
	
	Combining \eqref{eq:estimate:H2:Omega}, \eqref{eq:estimate:H2:Gamma} and \eqref{eq:estimate:pnu} together with estimate \eqref{prop:weak:solutions:01}, we get
	\begin{align}
		\label{eq:estimate:H2:Omega:02}
		\|u\|_{L^2(0,T;H^2(\Omega))}\leq C\|(f,f_\Gamma)\|_{L^2(0,T;L^2(\Omega))} + C\|(u_0,u_{\Gamma,0})\|_{\mathbb{H}^1},
	\end{align}   
	where we have chosen $\varepsilon>0$ small enough. With the estimate \eqref{eq:estimate:H2:Omega:02} at hand, by \eqref{eq:estimate:pnu} we can assert that
	\begin{align}
		\label{eq:estimate:pnu:02}
		\|\pnu u\|_{L^2(0,T;L^2(\Gamma))}\leq C\|(f,f_\Gamma)\|_{L^2(0,T;\mathbb{L}^2)} + C\|(u_0,u_{\Gamma,0})\|_{\mathbb{H}^1}.
	\end{align}
	
	Thus, substituting \eqref{eq:estimate:pnu:02} into \eqref{eq:estimate:H2:Gamma}, we conclude that
	\begin{align}
		\label{eq:estimate:H2:Gamma:02}
		\|u_\Gamma\|_{L^2(0,T;H^2(\Gamma))}\leq C\|(f,f_\Gamma)\|_{L^2(0,T;\mathbb{L}^2)} + C\|(u_0,u_{\Gamma,0})\|_{\mathbb{H}^1}.
	\end{align} 
	
	Finally, combining \eqref{eq:estimate:H2:Omega}, \eqref{eq:estimate:H2:Gamma:02} and \eqref{eq:estimate:L2:ptu01}, we deduce \eqref{estimate:H2}.
\end{proof}


\begin{proposition}
	\label{prop:extra:reg:epsilon:T}
	Let $(f,f_\Gamma)\in L^2(0,T;\mathbb{L}^2)$ and $(u_0,u_{\Gamma,0})\in \mathbb{L}^2$. Then, the associated weak solution $(u,u_\Gamma)$ satisfies
	\begin{align*}
		\begin{split}
			&\|(\sqrt{t}\pt u,\sqrt{t}\pt u_\Gamma)\|_{L^2(0,T;\mathbb{L}^2)}+ \|(\sqrt{t}u,\sqrt{t}u_\Gamma)\|_{L^2(0,T;\mathbb{H}^2)} + \|(\sqrt{t}u,\sqrt{t}u_\Gamma)\|_{C^0([0,T];\mathbb{H}^1)}\\
			\leq & C \|(f,f_\Gamma)\|_{L^2(0,T;L^2(\Omega))} + C\|(u_0,u_{\Gamma,0})\|_{\mathbb{L}^2}.
		\end{split}
	\end{align*}
\end{proposition}
\begin{proof}
	The proof is a slight modification of the arguments used in of Proposition \ref{prop:estimate:H2}. For this reason, we omit the details.
\end{proof}
\begin{remark} 
We point out that Proposition \ref{prop:extra:reg:epsilon:T} implies that, for each $T>0$, $(u_0,u_{\Gamma,0})\in \mathbb{L}^2$, $(f,f_\Gamma)\in L^2(0,T;\mathbb{L}^2)$ and $\varepsilon>0$, the weak solution $(u,u_\Gamma)$ of \eqref{linearized:problem:01} satisfies
\begin{align*}
	(u,u_\Gamma)\in H^1(\varepsilon,T;\mathbb{L}^2)\cap L^2(\varepsilon,T;\mathbb{H}^2)\cap C^0([\varepsilon,T];\mathbb{H}^1). 
\end{align*}
Moreover, there exists a constant $C>0$ (independent of $\varepsilon$) such that 
\begin{align*}
	&\|(u,u_\Gamma)\|_{H^1(\varepsilon,T;\mathbb{L}^2)} + \|(u,u_\Gamma)\|_{L^2(\varepsilon,T;\mathbb{H}^2)} + \|(u,u_\Gamma)\|_{C^0([\varepsilon,T];\mathbb{H}^1)}\\
	\leq & C\|(u_0,u_{\Gamma})\|_{\mathbb{L}^2} + C\|(f,f_\Gamma)\|_{L^2(0,T;\mathbb{L}^2)}.
\end{align*}
\end{remark}

now, we study local solutions of the nonlinear problem 
\begin{align}
	\label{prob:nonlinear:wp}
	\begin{cases}
		Lu+c(1+\gamma i)|u|^2 u=f,&\text{ in }\Omega\times (0,T),\\
		L_\Gamma(u,u_\Gamma)+ c(1+\gamma i)|u_\Gamma|^2 u_\Gamma=f_\Gamma ,&\text{ in }\Gamma\times (0,T),\\
		u=u_\Gamma,&\text{ on }\Gamma\times (0,T),\\
		(u(0),u_\Gamma (0))=(u_0,u_{\Gamma,0}),&\text{ in }\Omega\times \Gamma,
	\end{cases}
\end{align}
where $L$ and $L_\Gamma$ are given by  \eqref{def:L:N}, with $a,b,c>0$ and $\alpha,\gamma\neq 0$.
\begin{proposition}
Let $d=2$ or $d=3$. There exist $\varepsilon>0$ and $C>0$ such that, for every $(f,f_\Gamma)\in L^2(0,T;\mathbb{L}^2)$, $(u_0,u_{\Gamma,0})\in \mathbb{H}^1$ such that 
\begin{align}
	\label{assumption:fu0}
	\|(f,f_\Gamma)\|_{L^2(0,T;\mathbb{L}^2)} + \|(u_0,u_{\Gamma,0})\|_{\mathbb{H}^1}\leq \varepsilon,
\end{align}
there exists a unique solution $(u,u_\Gamma)$ of \eqref{prob:nonlinear:wp} which satisfies
\begin{align*}
	\|(u,u_\Gamma)\|_{C^0([0,T];\mathbb{H}^1)} + \|(u,u_\Gamma)\|_{L^2(0,T;\mathbb{H}^2)}\leq C\left(\|(f,f_\Gamma)\|_{L^2(0,T;\mathbb{L}^2)} + \|(u_0,u_{\Gamma,0})\|_{\mathbb{H}^1} \right). 
\end{align*}
\end{proposition}
\begin{proof}
We denote $\mathcal{B}=C^0([0,T];\mathbb{H}^1)\cap L^2(0,T;\mathbb{H}^2)$.
 Given $(f,f_\Gamma)\in L^2(0,T;\mathbb{L}^2)$ and $(u_0,u_{\Gamma,0})\in \mathbb{H}^1$,
for each   $(z,z_\Gamma)\in \mathcal{B}$,   
we consider the system
	\begin{align}
		\label{prob:nonlinear:wp2}
		\begin{cases} 
		Lu=f - c(1+\gamma i)|z |^2 z,&\text{ in }\Omega\times (0,T),\\
		L_\Gamma(u,u_\Gamma)=f_\Gamma - c(1+\gamma i)|z_\Gamma|^2 z_\Gamma,&\text{ on }\Gamma\times (0,T),\\
		u=u_\Gamma,&\text{ on }\Gamma\times (0,T),\\
		(u(0),u_\Gamma(0))=(u_0,u_{\Gamma,0}),&\text{ in } \Omega\times \Gamma.
		\end{cases}
	\end{align}

From Proposition \ref{prop:estimate:H2} and the fact that $\mathcal{B} \hookrightarrow  L^6(0,T;\mathbb{L}^6)$ continuously, we have that 
\eqref{prob:nonlinear:wp2} has a unique solution 
$(u, {u}_\Gamma) \in \mathcal B$. Hence, we can define the map $\mathcal{F}:\mathcal{B} \to \mathcal{B}$ given by $\mathcal{F}(z,z_\Gamma)=({u},{u}_\Gamma)$.
Moreover,  we also have that there exists $D>0$ such that 
\begin{align}\label{cotF}
	\|\mathcal{F}(z,z_\Gamma)\|_{\mathcal{B}} \leq D \left(\|(f,f_\Gamma)\|_{L^2(0,T;\mathbb{L}^2)} + \|(u_0,u_{\Gamma,0})\|_{\mathbb{H}^1} + \|(z,z_\Gamma)\|_{\mathcal{B}}^3 \right). 
\end{align}
	
%
	
	Clearly, $(u,u_\Gamma)\in \mathcal{B}$ is a solution of \eqref{prob:nonlinear:wp} if and only if it is a fixed point of the map $\mathcal{F}$. 
We will show that there exists $R>0$ such that the restriction of $\mathcal{F}$ to the closed ball $B_R:= \{(z,z_\Gamma)\in \mathcal{B}\,:\, \|(z,z_\Gamma)\|_{\mathcal{B}}\leq R\}$ is   a  contraction from $B_R$  into $B_R$. Then, the proof will follow from a classic fixed point result.

Indeed, from \eqref{cotF} and  assumption \eqref{assumption:fu0}, 
for each $(z,z_\Gamma) \in B_R$ we have 
\begin{align}
	\label{fixed:point:01}
	\|\mathcal{F}(z,z_\Gamma)\|_{\mathcal{B}}\leq D\left(\varepsilon +R^3\right).
\end{align}

Moreover, for each $(z,z_\Gamma),(w,w_\Gamma) \in B_R$, taking into account the equations satisfied by $\mathcal{F}(z,z_\Gamma)-\mathcal{F}(w,w_\Gamma)$, from Proposition \ref{prop:estimate:H2}
we have  that
\begin{align*}
	\|\mathcal{F}(z,z_\Gamma)-\mathcal{F}(w,w_\Gamma)\|_{\mathcal{B}}^2
	 \leq &  	 D_1 
 \| | z |^2 z - |w|^2w, |{z}_\Gamma |^2 z_\Gamma  - | w_\Gamma |^2w_\Gamma\|_{L^2(0,T;\mathbb{L}^2)}^2 	\\ 
\leq &	D_1 \left\{   \|z-w \|_{L^\infty(0,T; {L}^4(\Omega))}^2
\left( \|z^2 +zw  \|_{L^2(0,T; {L}^4(\Omega))}^2    + \|w \|_{L^4(0,T; {L}^8(\Omega))}^4   \right)  \right. \\
 &	\left. +  \|z_\Gamma-w_\Gamma \|_{L^\infty(0,T; {L}^4(\Gamma))}^2
\left( \| z_\Gamma^2 +z_\Gamma w_\Gamma  \|_{L^2(0,T; {L}^4(\Gamma))}^2    + \|w_\Gamma \|_{L^4(0,T; {L}^8(\Gamma))}^4   \right) \right\},
\end{align*}
%
and then, taking into account that $\mathcal{B} \hookrightarrow  L^\infty(0,T;\mathbb{L}^4) \cap L^4(0,T;\mathbb{L}^8)$, we get that
\begin{align}
	\label{fixed:point:02}
	\|\mathcal{F}(z,z_\Gamma)-\mathcal{F}(w,w_\Gamma)\|_{\mathcal{B}}
	\leq D_2R^2\|(z,z_\Gamma) -(w,w_\Gamma)\|_{\mathcal{B}}.
\end{align}
Therefore, in order to conclude, we choose $R>0$ such that 
$R^2 <   \min \left\{ \dfrac{1}{2D}, \dfrac{1}{D_2}  \right\} $ 
and $\varepsilon>0$ such that $\varepsilon \leq \frac{R}{2D}$. 
From \eqref{fixed:point:01}, \eqref{fixed:point:02} and the  Banach Fixed point Theorem, we get the existence of a  unique fixed point of $\mathcal{F}$.
\end{proof}

\section{A new Carleman estimate for the linear Ginzburg-Landau equation with dynamic boundary conditions}
	\label{Section:Carleman}
	In this section, we deduce a Carleman estimate for the linear Ginzburg-Landau operator with dynamic boundary conditions.
		We start introducing the weight functions that we shall use. For this propose, we recall the following
	\begin{lemma}
	\label{lemma:function:eta0}
		Given a nonempty open set $\omega' \Subset \Omega$, there exists a function $\eta^0 \in C^2(\overline{\Omega})$ such that 
		\begin{align}
			\label{function:eta0}
			\eta^0>0,\text{ in }\Omega,\quad \eta^0=0,\text{ on }\Gamma,\quad |\nabla \eta^0|>0,\text{ in }\overline{\Omega\setminus \omega'}.
		\end{align}
	\end{lemma}

	Given $\omega'\Subset \Omega$, we take $\eta^0$ with respect to $\omega'$ as in Lemma \ref{lemma:function:eta0}. For $\lambda,m>1$, we define 
	\begin{align}
		\label{def:varphi}
		\varphi(x,t):=&(t(T-t))^{-1} \left(e^{2\lambda m\|\eta^0\|_\infty} -e^{\lambda (m\|\eta^0\|_\infty +\eta^0(x))} \right), \quad \forall (x,t)\in \overline{\Omega}\times (0,T),\\
		\label{def:xi}
		\xi(x,t):=&(t(T-t))^{-1}e^{\lambda (m\|\eta^0\|_\infty + \eta^0(x))},\quad \forall (x,t)\in \overline{\Omega}\times (0,T). 
	\end{align} 
	\begin{theorem}
	\label{thm:Carleman:estimate}
	Let $\omega \Subset \Omega$. Set $\omega'\Subset \omega$ and $\eta^0$ satisfying \eqref{function:eta0}. Define $\varphi$ and $\xi$ as in \eqref{def:varphi} and \eqref{def:xi}, respectively. Then, there exist constants $C,\lambda_0,s_0>0$ such that for all $\lambda \geq\lambda_0$ and $s\geq s_0$, we have
	\begin{align}
		\label{thm:Carleman:ineq}
		\begin{split}
			&\int_0^T\int_\Omega e^{-2s\varphi} \left(s^3\lambda^4 \xi^3 |v|^2 + s\lambda^2 \xi |\nabla v|^2 + s^{-1}\xi^{-1}|\pt v|^2 + s^{-1}\xi^{-1}|\Delta v|^2 \right) dxdt\\
			&+\int_0^T\int_\Gamma e^{-2s\varphi} (s^3\lambda^3 \xi^3|v_\Gamma|^2 + s\lambda \xi |\nabla_\Gamma v_\Gamma|^2 + s\lambda |\pnu v|^2 +s^{-1}\xi^{-1} |\pt v_\Gamma|^2 + s^{-1}\xi^{-1}|\Delta_\Gamma v_\Gamma|^2 )dSdt\\
			\leq & Cs^3\lambda^4\int_0^T\int_\omega e^{-2s\varphi}\xi^3|v|^2 dxdt+ C \int_0^T \int_\Omega e^{-2s\varphi} |L^*(v)|^2 dxdt \\
			&+ C\int_0^T \int_\Gamma e^{-2s\varphi} |L_\Gamma^*(v,v_\Gamma)|^2 dS dt,
		\end{split}
	\end{align}	
	for all $(v,v_\Gamma)\in H^1(0,T;\mathbb{L}^2)\cap L^2(0,T;\mathbb{H}^2)$, where 
	\begin{align}
		\label{def:L:N:star}
		L^*(v)=\pt v+a(1-\alpha i)\Delta v,\quad L_\Gamma^*(v,v_\Gamma):=\pt v_\Gamma -a(1-\alpha i)\pnu v +b(1-\alpha i)\Delta_\Gamma v_\Gamma.
	\end{align}
	\end{theorem}
	
	\begin{proof}
	For convenience, the proof has been divided into several steps:
	
		\noindent $\bullet$ {\bf Step 1.} In this step, we perform the first estimates of the conjugated variables. For simplicity, we consider $v\in C^\infty(\overline{\Omega}\times [0,T])$, $v_\Gamma=v$, $\lambda \geq \lambda_1 \geq 1$, and $s\geq s_0\geq 1$. define 
		\begin{align*}
			w:=&e^{-s\varphi}v,\quad\text{ in }\overline{\Omega}\times (0,T),\\
			f:=&e^{-s\varphi}(\pt v +a(1-\alpha i)\Delta v),\quad\text{ in }\Omega\times (0,T),\\
			f_\Gamma :=&e^{-s\varphi}(\pt v -a(1-\alpha i)\pnu v + b(1-\alpha i)\Delta_\Gamma v),\quad\text{ on }\Gamma\times (0,T).
		\end{align*}
		
		Straightforward computations show that 
		\begin{align*}
			\nabla \varphi =-\nabla \xi  =-\lambda \xi \nabla \eta^0,\quad \Delta \varphi =-\lambda^2 \xi |\nabla \eta^0|^2 -\lambda \xi \Delta \eta^0,\quad  \pnu \varphi=-\lambda \xi \pnu \eta^0,\text{ on }\Gamma\times (0,T)
		\end{align*}
		with $\pnu \eta^0 \leq c<0$ on $\Gamma$, for some constant $c>0$ and 
		\begin{align*}
			\nabla_\Gamma \varphi =\nabla_\Gamma  \xi =0,\quad \Delta_\Gamma \varphi =\Delta_\Gamma \xi =0,\quad \text{ on }\Gamma\times (0,T).
		\end{align*}
		Then, in $\Omega\times (0,T)$ we have the following identity:
		\begin{align}
			\label{identity:P:Omega:01}
			\begin{split}
				f=&\pt w +a(1-\alpha i)\Delta w +as^2\lambda^2 (1-\alpha i)|\nabla \eta^0|^2 \xi^2 w -as\lambda^2 (1-\alpha i)|\nabla \eta^0|^2 \xi w\\
				&-as\lambda (1-\alpha i)\Delta \eta^0 \xi w -2as\lambda (1-\alpha i)\xi \nabla \eta^0 \cdot \nabla w +s\pt \varphi w.
			\end{split}
		\end{align}
		
		On the other hand, on $\Gamma\times (0,T)$ we have 
		\begin{align}
		\label{identity:P:Gamma:01}
			\begin{split}
				f_\Gamma =&\pt w -a(1-\alpha i)\pnu w +as\lambda (1-\alpha i)\pnu \eta^0 \xi w +b(1-\alpha i)\Delta_\Gamma w\\
				&+s\pt \varphi w
			\end{split}
		\end{align}
		
		Now, we define
		\begin{align}
			\label{def:P1:P2}
			\begin{split} 
			P_1w:=&a(s^2\lambda^2 |\nabla \eta^0|^2 \xi^2 w +\Delta w)+a\alpha i(2s\lambda \xi \nabla \eta^0 \cdot \nabla w +(s\lambda^2 |\nabla \eta^0|^2 +s\lambda \Delta \eta^0)\xi w)\\
			&+s\pt \varphi w
			\\
			P_2 w:=&-a(2s\lambda \xi \nabla \eta^0 \cdot \nabla w +(s\lambda^2 |\nabla \eta^0|^2 +s\lambda \Delta \eta^0)\xi w)-a\alpha i (s^2\lambda^2 |\nabla \eta^0|^2 \xi^2 w +\Delta w)\\
			& + \pt w\\	
			Rw:=&f.
			\end{split}
		\end{align}
		and
		\begin{align}
		\label{def:P1:P2:Gamma}
		\begin{split} 
			P_{\Gamma,1}w:=&b\Delta_\Gamma w-\dfrac{2a^2}{b}\alpha is\lambda \pnu \eta^0 \xi w+s\pt \varphi w, \quad
			P_{\Gamma,2}w:=-\alpha bi\Delta_\Gamma w+\dfrac{2a^2}{b}s\lambda \pnu \eta^0 \xi w+\pt w,\\
			R_{\Gamma}w:=&f_\Gamma -a(1-\alpha i)s\lambda \pnu \eta^0 \xi w+a(1-\alpha i)\pnu w+\dfrac{2a^2}{b} (1-\alpha i)s\lambda \pnu \eta^0 \xi w.
		\end{split}
		\end{align}
		Then, \eqref{identity:P:Omega:01} and \eqref{identity:P:Gamma:01} can be written as 
		\begin{align}
			\label{P1:P2:R}
			&P_1 w+P_2 w=Rw,\quad\text{ in }\Omega\times (0,T),\\
			\label{PG1:PG2:RG} 
			&P_{\Gamma,1}w+P_{\Gamma,2}w=R_\Gamma w,\quad\text{ on }\Gamma\times (0,T).
		\end{align}
		
		Then, taking the $L^2(\Omega\times (0,T))$-norm in \eqref{P1:P2:R} and the $L^2(\Gamma\times (0,T))$-norm in \eqref{PG1:PG2:RG}, we obtain
		\begin{align*}
			\begin{split}
				&\|P_1 w\|_{L^2(\Omega\times (0,T))}^2 + \|P_2 w\|_{L^2(\Omega\times (0,T))}^2 + \|P_{\Gamma,1}w\|_{L^2(\Gamma\times (0,T))}^2 + \|P_{\Gamma,2}w\|_{L^2(\Gamma\times (0,T))}^2 \\
				&+2\langle P_1 w,P_2 w \rangle_{L^2(\Omega\times (0,T))} + 2\langle P_{\Gamma,1}w, P_{\Gamma,2}w \rangle_{L^2(\Gamma\times (0,T))}\\
				=& \|Rw\|_{L^2(\Omega\times (0,T))}^2 + \|R_\Gamma w\|_{L^2(\Gamma\times (0,T))}^2 .
			\end{split}
		\end{align*} 
		\noindent $\bullet$ {\bf Step 2.} In this step, we devote to compute the terms 
		\begin{align*}
			\langle P_1 w,P_2w \rangle_{L^2(\Omega\times (0,T))}=\sum_{j,k=1}^3 I_{jk},
		\end{align*}
		where we have used the notation $I_{jk}$ to denote the real $L^2$ inner product between the $j^{\text{th}}$-term of $P_1w$ and the $k^{\text{th}}$-term of $P_2w$, defined in \eqref{def:P1:P2}. 
		
		Firstly, the term $I_{11}$ can be written as 
		\begin{align*}
			I_{11}:=&-a^2 \Re \int_0^T\int_\Omega (s^2\lambda^2 |\nabla \eta^0|^2 \xi^2 w +\Delta w)(2s\lambda \xi \nabla \eta^0 \cdot \nabla \overline{w} + (s\lambda^2 |\nabla \eta^0|^2 +s\lambda \Delta \eta^0)\xi \overline{w}) dxdt\\
			=&I_{11}^{(1)} + I_{11}^{(2)} + I_{11}^{(3)} + I_{11}^{(4)}.
		\end{align*}
		Integration by parts yields
		\begin{align*}
			I_{11}^{(1)}:=&-2a^2 s^3\lambda^3 \Re \int_0^T\int_\Omega |\nabla \eta^0|^2 \xi^3 w \nabla \eta^0 \cdot \nabla \overline{w} dxdt\\
			=&2a^2 s^3\lambda^3 \Re \int_0^T\int_\Omega \nabla \left( |\nabla \eta^0 |^2 \right) \cdot \nabla \eta^0 \xi^3 |w|^2 dxdt + 6a^2 s^3\lambda^4 \int_0^T\int_\Omega |\nabla \eta^0|^4 \xi^3 |w|^2 dxdt\\
			&-I_{11}^{(1)} -2a^2s^3\lambda^3 \int_0^T\int_\Gamma |\nabla \eta^0|^2 \pnu \eta^0 \xi^3 |w|^2 dxdt +2a^2 s^3\lambda^3 \int_0^T\int_\Omega |\nabla \eta^0|^2 \Delta \eta^0 \xi^3 |w|^2 dxdt.
		\end{align*}
		Then, $I_{11}^{(1)}$ is given by
		\begin{align*}
			\begin{split}
			I_{11}^{(1)}=&a^2 s^3\lambda^3 \int_0^T\int_\Omega \left( \nabla (|\nabla \eta^0|^2)\cdot \nabla \eta^0 + |\nabla \eta^0|^2 \Delta \eta^0 \right)\xi^3 |w|^2 dxdt \\
			&+3a^2 s^3\lambda^4 \int_0^T\int_\Omega |\nabla \eta^0|^4 \xi^3 |w|^2 dxdt-a^2 s^3\lambda^3 \int_0^T\int_\Gamma (\pnu \eta^0)^3 xi^3 |w|^2 dS dt.
			\end{split}
		\end{align*}
		On the other hand, we notice that 
		\begin{align*}
			I_{11}^{(2)}:=-a^2 \int_0^T\int_\Omega (s^3\lambda^4 |\nabla \eta^0|^4 +s^3\lambda^3|\nabla \eta^0|^2 \Delta \eta^0) \xi^3 |w|^2 dxdt
		\end{align*}
		Integrating by parts in space, we have 
		\begin{align*}
			I_{11}^{(3)}:=&-a^2 \Re \int_0^T\int_\Omega \Delta w(2s\lambda \xi \nabla \eta^0 \cdot \nabla \overline{w})dxdt\\
			=&2a^2 s\lambda^2 \int_0^T\int_\Omega \xi |\nabla \eta^0 \cdot \nabla w|^2 dxdt +2a^2 s\lambda \Re \int_0^T\int_\Omega \xi \nabla w \cdot \nabla (\nabla \eta^0 \cdot \nabla \overline{w})dxdt\\
			&-2a^2s\lambda \int_0^T \int_\Gamma \xi \pnu \eta^0 |\pnu w|^2 dSdt.
		\end{align*}
		
		Using the identity
		\begin{align*}
			\nabla w \cdot \nabla (\nabla \eta^0 \cdot \nabla \overline{w})=\nabla^2 \eta^0 (\nabla w,\nabla w)+\dfrac{1}{2} \nabla \eta^0 \cdot \nabla (|\nabla w|^2),\quad \text{ in }\Omega\times (0,T),
		\end{align*}
		we notice that 
		\begin{align*}
			&2a^2 s\lambda \Re \int_0^T\int_\Omega \xi \nabla w\cdot \nabla (\nabla \eta^0 \cdot \nabla \overline{w})dxdt\\
			=&2a^2 s\lambda \int_0^T\int_\Omega \xi \nabla^2 \eta^0 (\nabla w,\nabla \overline{w}) dxdt-a^2s\lambda \int_0^T\int_\Omega \Delta \eta^0 \xi |\nabla w|^2 dxdt\\
			&-a^2 s\lambda^2 \int_0^T\int_\Omega \xi |\nabla \eta^0|^2 |\nabla w|^2 dxdt+a^2 s\lambda \int_0^T\int_\Gamma \pnu \eta^0 \xi (|\nabla_\Gamma w|^2 + |\pnu w|^2)dSdt.
		\end{align*}
		
		Then, $I_{11}^{(3)}$ is given by 
		\begin{align*}
			\begin{split}
				I_{11}^{(3)}=&2a^2 s\lambda^2 \int_0^T\int_\Omega \xi |\nabla \eta^0 \cdot \nabla w|^2 dxdt +2a^2 s\lambda \int_0^T\int_\Omega \xi \nabla^2 \eta^0 (\nabla w,\nabla \overline{w})dxdt\\
				&-a^2 s\lambda \int_0^T\int_\Omega \Delta \eta^0 \xi |\nabla w|^2 dxdt -a^2 s\lambda^2 \int_0^T\int_\Omega \xi |\nabla \eta^0|^2 |\nabla w|^2 dxdt\\
				&-a^2 s\lambda \int_0^T\int_\Gamma \pnu \eta^0 \xi \left( |\pnu w|^2 -|\nabla_\Gamma w|^2 \right) dS dt.
			\end{split}
		\end{align*}
		
		After integration by parts, $I_{11}^{(4)}$ can be written as follows:
		\begin{align*}
			\begin{split} 
			I_{11}^{(4)}
			:=&a^2 \Re \int_0^T\int_\Omega \xi \overline{w} \nabla w\cdot \nabla (s\lambda^2 |\nabla \eta^0|^2 + s\lambda \Delta \eta^0) dxdt\\
			 &+a^2 \Re \iint_0^T\int_\Omega \lambda\xi \left( s\lambda^3 |\nabla \eta^0|^2 +s\lambda^2 \Delta \eta^0 \right)\overline{w} \nabla \eta^0 \cdot \nabla w dxdt\\
			&+a^2 \iint_0^T\int_\Omega (s\lambda^2 |\nabla \eta^0|^2 +s\lambda \Delta \eta^0)\xi |\nabla w|^2 dxdt\\
			&-a^2 \Re \iint_0^T\int_\Gamma (s\lambda^2 |\pnu \eta^0|^2 +s\lambda \Delta \eta^0) \xi (\pnu w) \overline{w}dSdt.
			\end{split} 
		\end{align*}
		
		Then, by \eqref{function:eta0}, $I_{11}$ can be estimated as 
		\begin{align*}
			I_{11}\geq &C \int_0^T\int_\Omega  (s^3\lambda^4\xi^3|w|^2 +s\lambda^2  \xi |\nabla \eta^0 \cdot \nabla w|^2) dxdt\\
			&+Cs\lambda \int_0^T\int_\Gamma \xi |\pnu w|^2 dSdt + a^2s\lambda \int_0^T\int_\Gamma \pnu \eta^0 \xi |\nabla_\Gamma w|^2 dSdt\\
			&-Cs^3\lambda^4 \int_0^T\int_\omega \xi^3 |w|^2 dxdt -X_{11},
		\end{align*}
		where $X_{11}$ satisfies the following upper bound:
		\begin{align*}
			X_{11}\leq &C \int_0^T\int_\Omega (s^3\lambda^3\xi^3 |w|^2 +s\lambda \xi |\nabla w|^2) dxdt\\
			&+C \int_0^T\int_\Gamma \left( s^2\lambda^3 \xi^3 |w|^2 + \lambda  \xi |\pnu w|^2\right) dSdt.
			\end{align*}
		
		It is clear that 
		\begin{align*}
			I_{12}
			=&0.
		\end{align*}
		
		Moreover, we write $I_{13}$ as
		\begin{align*}
			I_{13}
			:=&as^2\lambda^2\Re \int_0^T\int_\Omega |\nabla \eta^0|^2 \xi^2 w\pt \overline{w} dxdt +a\Re \int_0^T\int_\Omega \Delta w \pt \overline{w} dxdt\\
			=&I_{13}^{(1)} + I_{13}^{(2)}. 
		\end{align*}
		Integrating by parts on time and using the fact that $w=0$ in $t=0$ and $t=T$, we have
		\begin{align*}
			I_{13}^{(1)}
			=&-as^2\lambda^2 \int_0^T\int_\Omega |\nabla \eta^0|^2 \xi \pt \xi |w|^2 dxdt.
		\end{align*}
		
		In the same manner, $I_{13}^{(2)}$ gives
		\begin{align*}
			I_{13}^{(2)}
			=&a\Re \int_0^T\int_\Gamma \pnu w \pt \overline{w} dSdt.
		\end{align*}
		
		Then, $I_{13}$ is given by 
		\begin{align*}
			I_{13}
			\geq & -Cs^2\lambda^3 \int_0^T \int_\Omega \xi^3 |w|^2 dxdt-C\int_0^T\int_{\Gamma} (s^{1/2}\xi |\pnu w|^2-s^{-1/2} \xi^{-1} |\pt w|^2) dSdt.
		\end{align*}
		
		Moreover, from the definition of $I_{21}$, we see that  
		\begin{align*}
			I_{21}
			=&0
		\end{align*}
		Similar to the computations of $I_{11}$, $I_{22}$ can be estimated as follows:
		\begin{align*}
			I_{22}\geq &C \int_0^T\int_\Omega \left(s^3\lambda^4\xi^3 |w|^2 +s\lambda^2 \xi|\nabla \eta^0 \cdot \nabla w|^2 \right) dxdt\\
			&+C\int_0^T\int_\Gamma (s^3\lambda^3 \xi^3 |w|^2 +Cs\lambda \xi |\pnu w|^2+a^2\alpha^2 s\lambda \pnu \eta^0 \xi |\nabla_\Gamma w|^2) dSdt\\
			&-Cs^3\lambda^4 \int_0^T\int_\omega\xi^3 |w|^2 dSdt -X_{22},
		\end{align*}
		where $X_{22}$ satisfies
		\begin{align*}
			X_{22}\leq &C \int_0^T\int_\Omega (s^3\lambda^3\xi^3 |w|^2 +Cs\lambda \xi |\nabla w|^2)dxdt\\
			&+C \int_0^T\int_{\Gamma} \left( s^2\lambda^3 \xi^2 |w|^2 + \lambda \xi |\pnu w|^2\right) dSdt.
		\end{align*}
%
%
		To compute the term $I_{23}$, we follow \cite{Rosier2009}. Then, we write 
		\begin{align}
			\nonumber 
			I_{23}
			\nonumber 
			=&\dfrac{1}{2}a\alpha i \int_0^T\int_\Omega (2s\lambda \xi \nabla \eta^0 \cdot \nabla w + (s\lambda^2 |\nabla \eta^0|^2 +s\lambda \Delta \eta^0)\xi w)\pt \overline{w}dxdt\\
			\nonumber 
			&-\dfrac{1}{2}a\alpha i \int_{0}^T\int_\Omega (2s\lambda \xi \nabla \eta^0 \cdot \nabla \overline{w} +(s\lambda^2 |\nabla \eta^0|^2 + s\lambda \Delta \eta^0)\xi \overline{w})\pt w dxdt\\
			\label{estimate:I23}
			=&I_{23}^{(1)} + I_{23}^{(2)}.
		\end{align} 
		
		On one hand, integrating by parts in time we have 
		\begin{align}
			\label{estimate:I23:01}
			\begin{split} 
			I_{23}^{(1)}:=&-\dfrac{1}{2} a\alpha i \int_0^T\int_\Omega (2s\lambda \pt \xi \nabla \eta^0 \cdot \nabla w +2s\lambda \xi \nabla \eta^0 \cdot \nabla \pt w)\overline{w} dxdt\\
			&-\dfrac{1}{2} a\alpha i \int_0^T\int_\Omega \left( (s\lambda^2 |\nabla \eta^0|^2 + s\lambda \Delta \eta^0)\pt \xi |w|^2 +(s\lambda^2|\nabla \eta^0|^2 +s\lambda \Delta \eta^0 )\xi \overline{w}\pt w \right) dxdt.
			\end{split}
		\end{align}
		
		On the other hand, integration by parts in space yields
		\begin{align}
			\label{estimate:I23:02}
			\begin{split} 
			I_{23}^{(2)}=&\dfrac{1}{2}a\alpha i \int_0^T\int_\Omega  \left( 
			s\lambda^2 |\nabla \eta^0|^2 \xi \overline{w}\pt w +s\lambda \xi \overline{w} \nabla \eta^0 \cdot \nabla \pt w + s\lambda \Delta \eta^0 \xi \overline{w}\pt w \right) dxdt\\
			&-\dfrac{1}{2}a\alpha i \int_0^T\int_\Omega (s\lambda^2 |\nabla \eta^0|^2 +s\lambda \Delta \eta^0)\xi \overline{w}\pt w dxdt -a\alpha s\lambda i \int_0^T\int_\Gamma \pnu \eta^0 \xi \overline{w}\pt w dSdt. 
			\end{split}
		\end{align}
		
		Then, substituting \eqref{estimate:I23:01} and \eqref{estimate:I23:02} into \eqref{estimate:I23} and applying the Young's inequality, we show that for all $\epsilon>0$, there exists $C(\epsilon)>0$ such that 
		\begin{align*}
			I_{23}\geq &-C\int_0^T\int_\Omega \left( s^2\lambda^2 \xi^3 |w|^2 +\lambda \xi |\nabla w|^2  \right)dxdt\\
			&-\int_0^T\int_{\Gamma} (\epsilon s^3\lambda^3 \xi^3 |w|^2+C(\epsilon)s^{-1}\lambda^{-1}\xi^{-1} |\pt w|^2)dSdt. 
		\end{align*}
		
		The term $I_{31}$ is
		\begin{align*}
			I_{31}=&-2as^2\lambda \Re \int_0^T\int_\Omega \xi \pt \varphi w\nabla \eta^0 \cdot \nabla \overline{w} dxdt -a\int_0^T\int_\Omega (s^2\lambda^2 |\nabla \eta^0|^2 +s^2 \lambda \Delta \eta^0)\pt \varphi \xi |w|^2 dxdt
		\end{align*}
		where the first term can be computed as
		\begin{align*}
			&-2as^2\lambda \Re \int_0^T\int_\Omega \xi \pt \varphi w\nabla \eta^0 \cdot \nabla \overline{w} dxdt\\
			=& as^2\lambda^2 \int_0^T\int_\Omega (|\nabla \eta^0|^2 \xi \pt \varphi +\xi \nabla \eta^0 \cdot \nabla (\pt \varphi) -\Delta \eta^0 \xi \pt \varphi )|w|^2 dxdt\\
			&-2as^2\lambda \int_0^T\int_\Gamma \pnu \eta^0 \xi \pt \varphi |w|^2 dSdt
		\end{align*}
		
		Then, $I_{31}$ is estimated as 
		\begin{align*}
			I_{31}
			\geq & -Cs^2\lambda^2\int_0^T\int_\Omega \xi^3 |w|^2 dxdt-Cs^2\lambda^2\int_0^T\int_\Gamma \xi^3 |w|^2 dSdt. 
		\end{align*}
		
		Now, for $I_{32}$, it is clear that
		\begin{align*}
			I_{32}
			=&-a\alpha s\Im \int_0^T\int_\Omega w\nabla (\pt \varphi) \cdot \nabla \overline{w} dxdt +a\alpha s\Im \int_0^T\int_\Gamma \pt \varphi w\pnu \overline{w} dSdt\\
			\geq & -C\int_0^T\int_\Omega \left( \xi^3 |w|^2+\xi |\nabla w|^2 \right) dxdt\\
			&-C\int_0^T\int_\Gamma \left(s\xi^3 |w|^2 + s\xi |\pnu w|^2 \right) dSdt.
		\end{align*}
		
		In addition, since $w=0$ in $t=0$ and $t=T$, we have 
		\begin{align*}
			I_{33}\geq -Cs\int_0^T\int_\Omega \xi^3 |w|^2 dxdt.
		\end{align*}
		
		According to the above estimates, taking $s_1$ and $\lambda_1>0$ large enough if it is neccesary and choosing $\epsilon>0$ small enough, we deduce that
		\begin{align}
			\label{product:omega}
			\begin{split}
				&\langle P_1 w,P_2w \rangle_{L^2(\Omega\times (0,T))}\\
				 \geq &Cs^3\lambda^4 \int_0^T\int_\Omega \xi^3 |w|^2 dxdt+C\int_0^T\int_\Gamma (s^3\lambda^3 \xi^3 |w|^2 +s\lambda \xi |\pnu w|^2) dSdt\\
				&+a^2(1+\alpha^2)s\lambda \int_0^T\int_\Gamma \pnu \eta^0 \xi |\nabla_\Gamma w|^2 dSdt-Cs^3\lambda^4 \int_0^T\int_\omega \xi^3 |w|^2 dxdt-\tilde{X},
			\end{split}
		\end{align}
		for $s\geq s_1 >0$ and $\lambda \geq \lambda_1 >0$, where $X$ satisfies 
		\begin{align*}
			X\leq Cs\lambda \int_0^T\int_\Omega \xi |\nabla w|^2 dxdt+Cs^{-1}\lambda^{-1}\int_0^T\int_\Gamma \xi^{-1} |\pt w|^2 dSdt.
		\end{align*}
		
		Before going any further, let us point out that the integral term $|\nabla_\Gamma w|^2$ is negative. However, in the next step we obtain additional terms to solve this problem.
		 
		\noindent $\bullet$ {\bf Step 3.} In this step, we compute the terms
		\begin{align*}
			\int_0^T\int_\Gamma P_{\Gamma,1} w\overline{P_{\Gamma,2}w}dSdt=\sum_{j=1}^3 \sum_{k=1}^3 J_{jk},
		\end{align*}
		where $J_{jk}$ denotes the real $L^2$ inner product of the j$^{\text{th}}$-term of $P_{1,\Gamma} w$ with the k$^{th}$-term of $P_{2,\Gamma} w$ defined in \eqref{def:P1:P2:Gamma}.
		Firstly, we have  
		\begin{align*}
			J_{11}=&b^2\alpha \Im \int_0^T\int_\Gamma |\Delta_\Gamma w|^2 dSdt=0.
		\end{align*}
		Secondly, by the surface divergence theorem, we get
		\begin{align*}
			J_{12}=&2a^2s\lambda \Re \int_0^T\int_\Gamma (\xi \overline{w} \nabla_\Gamma (\pnu \eta^0)\cdot \nabla_\Gamma w)dSdt - 2a^2s\lambda \int_0^T\int_\Gamma \pnu \eta^0 \xi |\nabla_\Gamma w|^2 dSdt\\
			\geq & -2a^2s\lambda \int_0^T\int_\Gamma \pnu \eta^0 \xi |\nabla_\Gamma w|^2 dSdt-C\int_0^T\int_\Gamma (s^2\lambda^2 \xi |w|^2 + \xi |\nabla_\Gamma w|^2) dSdt.
		\end{align*}
		
		Moreover, integrating by parts in time and space, we can assert that
		\begin{align*}
			J_{13}
			=&0.
		\end{align*}		
		
		The term $J_{21}$ is
		\begin{align*}
			J_{21}
			=&-2a^2\alpha^2 s\lambda \Re\int_0^T\int_\Gamma \xi w\nabla_\Gamma (\pnu \eta^0)\cdot \nabla_\Gamma \overline{w} dSdt-2a^2\alpha s\lambda \int_0^T\int_\Gamma \pnu \eta^0 \xi |\nabla_\Gamma w|^2 dSdt\\
			\geq &-2a^2\alpha^2 s\lambda \int_0^T\int_\Gamma \pnu \eta^0 \xi |\nabla_\Gamma w|^2 dSdt-C\int_0^T\int_\Gamma (s^2\lambda^2\xi |w|^2 +\xi |\nabla_\Gamma w|^2)dSdt.
		\end{align*}
		
		Furthermore, by definition, $J_{22}$ is
		\begin{align*}
			J_{22}=0
		\end{align*}
		
		For $J_{23}$, we use Young's inequality to show that for all $\epsilon>0$, there exists $C(\epsilon)>0$ such that
		\begin{align*}
			J_{23}
			\geq & -\int_0^T\int_\Gamma \left(\epsilon s^{3}\lambda^{3}\xi^3 |w|^2 + C(\epsilon) s^{-1}\lambda^{-1} \xi^{-1} |\pt w|^2 \right) dSdt.
		\end{align*}
		
		By the surface divergence theorem and the fact that $\nabla_\Gamma \varphi =0$ on $\Gamma\times (0,T)$, we deduce that 
		\begin{align*}
			J_{31}
			=&0.
		\end{align*}
		
		Moreover, by definition $J_{32}$ is given by
		\begin{align*}
			J_{32} \geq -Cs^2\lambda \int_0^T\int_\Gamma \xi^3 |w|^2 dSdt.
		\end{align*}
		
		On the other hand, integration by parts yields
		\begin{align*}
			J_{33}
			\geq & -Cs\int_0^T\int_\Gamma \xi^3 |w|^2 dSdt.
		\end{align*}

		According to these estimates, we can assert that 
		\begin{align}
		\label{product:boundary}
		\begin{split}
			\langle P_{\Gamma,1} w,P_{\Gamma,2} w \rangle_{L^2(\Gamma\times (0,T))} \geq -2a^2(1+\alpha^2) s\lambda \int_0^T\int_\Gamma \pnu \eta^0 \xi |\nabla_\Gamma w|^2 dSdt - Y, 
		\end{split}
		\end{align}
		where $Y$ satisfies
		\begin{align*}
			Y\leq &\int_0^T\int_\Gamma (\epsilon s^{3}\lambda^{3} \xi^3 |w|^2 +C(\epsilon)s^{-1} \lambda^{-1} \xi^{-1} |\pt w|^2) dSdt\\
			&+C\int_0^T\int_\Gamma \xi |\nabla_\Gamma w|^2 dSdt.
		\end{align*}
		
		From \eqref{product:omega} and \eqref{product:boundary}, using that $\pnu \eta^0 \leq -c <0$ on $\Gamma \times (0,T)$, taking $s_1,\lambda_1>0$ large enough and choosing $\epsilon>0$ small enough, we obtain 

		\begin{align}
			\label{estimate:incomplete}
			\begin{split}
				&\|P_1 w\|_{L^2(\Omega\times (0,T))}^2 + \|P_2 w\|_{L^2(\Omega\times (0,T))}^2 + \|P_{\Gamma,1} w\|_{L^2(\Gamma\times (0,T))}^2+ \|P_{\Gamma,2}w\|_{L^2(\Gamma\times (0,T))}^2\\
				&+C\int_0^T\int_\Omega (s^3\lambda^4\xi^3 |w|^2 +s\lambda^2\xi |\nabla \eta^0 \cdot \nabla w|^2)dxdt\\
				&+C\int_0^T\int_\Gamma (s^3\lambda^3\xi^3 |w|^2 +s\lambda \pnu \eta^0 \xi |\nabla_\Gamma w|^2) dSdt\\
				\leq & \|f\|_{L^2 (\Omega\times (0,T))}^2 + \|f_\Gamma \|_{L^2(\Gamma\times (0,T))} ^2 + s^3\lambda^4\int_0^T\int_\omega \xi^3 |w|^2 dxdt  +Z,
			\end{split}
		\end{align}
		where $Z$ is given by 
		\begin{align*}
			Z\leq &C\int_0^T\int_{\Gamma} (s^{-1/2}\lambda^{-1/2}+s^{-1})\xi^{-1} |\pt w|^2 dSdt\\
			&+Cs\lambda \int_0^T\int_\Omega \xi |\nabla w|^2 dxdt.
		\end{align*} 
		
		In order to absorb the terms of $|\nabla v|$ in $\Omega\times (0,T)$, $\Delta w$ and $\pt w$ on $\Gamma\times (0,T)$, we shall use an indirect estimates using the definitions given in \eqref{P1:P2:R} and \eqref{PG1:PG2:RG}.
		
		\noindent $\bullet$ {\bf Step 4.} In this step, we obtain estimates for the $L^2$-norm of $\Delta w$, $\nabla w$ and $\pt w$ in $\Omega\times (0,T)$. More precisely, the purpose of this step is to prove the following inequality:
		\begin{align}
			\label{estimate:omega:all:01}
			\begin{split}
			&\int_0^T\int_\Omega \left( s^{-1}\xi^{-1} |\Delta w|^2 + s^{-1}\xi^{-1} |\pt w|^2 + s\lambda^2 |\nabla w|^2 \right) dxdt\\
			\leq & C \int_0^T\int_\Omega \left( s^3\lambda^4 \xi |w|^2 + s\lambda^2 \xi |\nabla \eta^0 \cdot \nabla w|^2 +s^{-1}\xi^{-1} |P_1 w|^2 + s^{-1}\xi^{-1} |P_2 w|^2 \right) dxdt\\
			&+C\int_0^T\int_\Gamma \left( s^3\lambda^3 |w|^2 + s\lambda |\pnu w|^2 \right) dSdt.
			\end{split}
		\end{align}
		
		In order to do that, we firstly estimate the term of $\Delta w$. From the definition of $P_1$ in \eqref{def:P1:P2}, it is clear that
		\begin{align}
			\label{estimate:Delta:w:omega}
			\begin{split}
			&s^{-1}\int_0^T\int_\Omega  \xi^{-1} |\Delta w|^2 dxdt\\
			\leq &C \int_0^T\int_\Omega \left(  s^{-1}\xi^{-1} |P_1 w|^2 + s^3\lambda^4 |w|^2 +s\lambda^2  \xi |\nabla \eta^0 \cdot \nabla w|^2 \right) dxdt.
			\end{split}
		\end{align} 
		
		Secondly, we estimate $\nabla w$ as follows: 
		\begin{align*}
			s\lambda^2 \int_0^T\int_\Omega \xi |\nabla w|^2 dxdt =&-s\lambda^2 \Re \int_0^T\int_\Omega w\nabla \xi \cdot \nabla \overline{w} dxdt +s\lambda^2 \Re \int_0^T\int_\Gamma \xi w \pnu \overline{w} dSdt\\
			&-s\lambda^2 \Re \int_0^T\int_\Omega \xi w \Delta \overline{w} dxdt\\
		\end{align*}
		Then, by Young's inequality we obtain
		\begin{align}	 
			\label{estimate:nabla:w:omega}
			\begin{split}
			s\lambda^2 \int_0^T\int_\Omega \xi |\nabla w|^2 dxdt
			\leq & C \int_0^T\int_\Omega \left( s^3\lambda^4 \xi^3 |w|^2 +s\lambda^2 \xi |\nabla \eta^0 \cdot \nabla w|^2 + s^{-1}\xi^{-1} |\Delta w|^2 \right) dxdt \\
			&+C\int_0^T\int_\Gamma (s^3\lambda^3 |w|^2 +s\lambda |\pnu w|^2)dSdt.
			\end{split}
		\end{align} 
		
		Next, we estimate the term of $\pt w$ directly from the definition of $P_2$:
		\begin{align}
		\label{estimate:pt:w:omega}
			s^{-1}\int_0^T\int_\Omega \xi^{-1} |\pt w|^2 dxdt \leq C \int_0^T\int_\Omega \left( s^3\lambda^4 \xi |w|^2 + s\lambda^2 \xi |\nabla \eta^0\cdot \nabla w|^2 + s^{-1}\xi^{-1} |\Delta w|^2 \right) dxdt.
		\end{align}
		
		Thus, combining \eqref{estimate:Delta:w:omega}, \eqref{estimate:nabla:w:omega} and \eqref{estimate:pt:w:omega}, we easily get \eqref{estimate:omega:all:01}.
		
		\noindent $\bullet$ {\bf Step 5} In this step, we shall prove that
			\begin{align}
			\label{estimate:ptw:all:02}
			\begin{split}
			&s^{-1}\int_0^T\int_\Gamma \left( \xi^{-1} |\pt w|^2 + \xi^{-1} |\Delta_\Gamma w|^2 \right) dSdt\\
			\leq & C \int_0^T\int_\Gamma \left(s^2\lambda^2 |w|^2 + \xi^{-1} |P_{\Gamma,1}w|^2 +\xi^{-1}|P_{\Gamma,2} w|^2 \right) dSdt.
			\end{split}
		\end{align}
		
		From one hand, by definition of $P_{\Gamma,1}$ in \eqref{def:P1:P2:Gamma}, we have 
		\begin{align}
			\label{estimate:boundary:01}
			s^{-1}\int_0^T\int_\Gamma \xi^{-1} |\Delta_\Gamma w|^2 dSdt\leq C\int_0^T\int_\Gamma \left( |P_{\Gamma,1} w|^2 + s^2\lambda^2 \xi^3 |w|^2 \right) dSdt.
		\end{align} 
		
		On the other hand, using the definition of $P_{\Gamma,2}$ we deduce that 
		\begin{align}
			\label{estimate:boundary:02}
			s^{-1}\int_0^T\int_\Gamma \xi^{-1} |\pt w|^2 dSdt\leq C \int_0^T\int_\Gamma \left(s^{-1} \xi^{-1} |\Delta_\Gamma w|^2 + s\lambda \xi |w|^2+ |P_{\Gamma,2} w|^2  \right) dSdt.
		\end{align}
		
		Combining \eqref{estimate:boundary:01} and \eqref{estimate:boundary:02}, we obtain \eqref{estimate:ptw:all:02}.

		Finally, combining \eqref{estimate:omega:all:01}, \eqref{estimate:ptw:all:02} and \eqref{estimate:incomplete}, and taking $s,\lambda >0$, we obtain
		\begin{align*}
			\begin{split}
				&\int_0^T \int_\Omega (s^3 \lambda^4 \xi^3 |w|^2 + s\lambda^2 \xi |\nabla w|^2 + s^{-1} \xi^{-1} |\pt w|^2 + s^{-1} \xi^{-1} |\Delta w|^2) dxdt\\
				&+\int_0^T \int_\Gamma \left(s^3\lambda^3 \xi^3 |w|^2 + s\lambda |\pnu w|^2 + s\lambda |\nabla_\Gamma w|^2 + s^{-1}\xi^{-1}|\Delta_\Gamma w|^2 +s^{-1}\xi^{-1}|\pt w|^2 \right) dSdt\\
				\leq & C\|Rw\|_{L^2(\Omega\times (0,T))}^2 +C \|R_\Gamma w\|_{L^2(\Gamma\times (0,T))}^2+ Cs^3\lambda^4 \int_0^T\int_\omega \xi^3 |w|^2 dxdt.
			\end{split}
		\end{align*}
		Finally, taking into account that $w=e^{-s\varphi}$, we come back to the original variable and conclude the inequality \eqref{thm:Carleman:ineq}.
  	\end{proof}

\section{Observability and null controllability of the linear system} 
\label{section:Observability} 	
The goal of this section is to prove a null controllability result for the linear Ginzburg-Landau
\begin{align}
	\label{linear:system:y}
	\begin{cases}
		L(y)=f+\mathbbm{1}_\omega h,&\text{ in }\Omega\times (0,T),\\
		L_\Gamma(y,y_\Gamma)=f_\Gamma,&\text{ in }\Gamma\times (0,T),\\
		y=y_\Gamma,&\text{ on }\Gamma \times (0,T),\\
		(y(0),y_\Gamma (0))=(y_0,y_{\Gamma,0}),&\text{ in }\Omega\times \Gamma.
	\end{cases}
\end{align}  
where the operators $L$ and $L_\Gamma$ were defined in \eqref{def:L:N}. Naturally, the null controllability for \eqref{linear:system:y} can be expressed in the following terms:
\begin{definition}
	We say that \eqref{linear:system:y} is null controllable in $\mathbb{L}^2$ if for all $T>0$, $(y_0,y_{\Gamma,0})\in \mathbb{L}^2$, there exists a control $h\in L^2(\omega\times (0,T))$ such that the associated solution $(y,y_\Gamma)$ of \eqref{linear:system:y} satisfies
	\begin{align*}
		y(\cdot,T)=0,\text{ in }\Omega,\quad y_\Gamma(\cdot,T)=0,\text{ on }\Gamma.
	\end{align*} 
\end{definition}

Following the classical duality between controllability and observability in the context of parabolic equations (see e.g. \cite{fursikov1996controllability}), the null controllability of \eqref{linear:system:y} is equivalent to prove a suitable observability inequality for its adjoint system. Then, thanks to this inequality and the Lax-Milgram lemma, we allow us to deduce the null controllability of \eqref{linear:system:y}, which is crucial for the proof of the Theorem \eqref{main:thm} and an interesting result by itself.
\subsection{Observability inequality}

As we explained above, we introduce the adjoint system
\begin{align}
	\label{linear:system:z}
	\begin{cases}
		L^*z=g,&\text{ in }\Omega\times (0,T),\\
		L_\Gamma^*(z,z_\Gamma) z=g_\Gamma,&\text{ on }\Gamma \times (0,T),\\
		z=z_\Gamma,&\text{ on }\Gamma\times (0,T),\\
		(z(T),z_\Gamma (T))=(z_T,z_{\Gamma,T}),&\text{ in }\Omega\times \Gamma.
	\end{cases}
\end{align}
where $L^*$ and $L_\Gamma^*$ are given by \eqref{def:L:N:star}. We point out that, if $(z_T,z_{\Gamma,T})\in \mathbb{L}^2$, then the associated weak solution $(z,z_\Gamma)$ belongs to $C^0([0,T];\mathbb{L}^2)\cap L^2(0,T;\mathbb{H}^1)$. This is done by results of Section \ref{section:Existence}.

To formulate the result of this subsection, we shall define some additional functions. For $t\in (0,T)$, we define
\begin{align*}
	\mu(t):=&\begin{cases}
		\dfrac{4}{T^2},\quad \text{ if }t\in (0,T/2],\\
		\dfrac{1}{t(T-t)},\quad \text{ if }t\in (T/2,T),
	\end{cases}\\
	\check{\varphi}(t):=&\mu(t)\min_{x\in \overline{\Omega}} \left(e^{2s\lambda m\|\eta^0\|_\infty} - e^{\lambda (m\|\eta^0\|_\infty +\eta^0(x))} \right),\quad t\in (0,T), \\
	\hat{\varphi}(t):=&\mu(t)\max_{x\in \overline{\Omega}} \left(e^{2s\lambda m\|\eta^0\|_\infty} - e^{\lambda (m\|\eta^0\|_\infty +\eta^0(x))} \right),\quad t\in (0,T),\\
	\check{\xi}(t):=&\mu(t)\min_{x\in\overline{\Omega}} e^{\lambda (m\|\eta^0\|_\infty +\eta^0(x))},\quad t\in (0,T),\\
	\hat{\xi}(t):=&\mu(t)\max_{x\in \overline{\Omega}} e^{\lambda (m\|\eta^0\|_\infty + \eta^0(x))},\quad t\in (0,T).
\end{align*}

\begin{proposition}[Observability inequality]
	\label{proposition:observability:inequality}
	Let $d\geq 2$. Suppose that  $(z_T,z_{\Gamma,T})\in \mathbb{L}^2$ and $(g,g_\Gamma)\in L^2(0,T;\mathbb{L}^2)$. Then, there exists a constant $C>0$ such that the associated weak solution $(z,z_\Gamma)\in C^0([0,T];\mathbb{L}^2)\cap L^2(0,T;\mathbb{H}^1)$ of \eqref{linear:system:z} satisfies 
	\begin{align}
		\label{observability:inequality}
		\begin{split}
			&\int_\Omega |z(0)|^2 dx + \int_\Gamma |z_\Gamma(0)|^2 dS + \int_0^T \int_\Omega e^{-2s\hat{\varphi}} (\check{\xi}^3|z|^2 + \check{\xi}|\nabla z|^2) dxdt\\
			&+\int_0^T\int_\Gamma e^{-2s\hat{\varphi}} (\check{\xi}^3|z_\Gamma|^2 + \check{\xi}|\nabla_\Gamma z_\Gamma|^2) dSdt\\
			\leq & C\int_0^T \int_\Omega e^{-2s\check{\varphi}} |g|^2 dxdt+ C\int_0^T \int_\Gamma e^{-2s\check{\varphi}} |g_\Gamma|^2 dSdt+ C\int_0^T\int_\omega e^{-2s\check{\varphi}}\hat{\xi}^3 |z|^2 dxdt.
		\end{split}
	\end{align}
\end{proposition}

\begin{proof}
	Consider a cut-off function $\theta\in C^1([0,T];\mathbb{R})$ such that 
	\begin{align*}
		0\leq \theta(t)\leq 1,\quad \forall t\in [0,T],\quad \theta(t)=1,\quad \forall t\in [0,T/2],\quad \theta(t)=0,\quad \forall t\in [3T/4,T]. 
	\end{align*}
	Then, the new variables $(\tilde{z},\tilde{z}_\Gamma)=(\theta z,\theta z_\Gamma)$ satisfies 
	\begin{align*}
		\begin{cases}
			L^*\tilde{z}=\theta g+\theta' z,&\text{ in }\Omega\times (0,T),\\
			L_\Gamma^*(\tilde{z},\tilde{z}_\Gamma)=\theta g_\Gamma +\theta'z_\Gamma,&\text{ on }\Gamma\times (0,T),\\
			\tilde{z}=\tilde{z}_\Gamma,&\text{ on }\Gamma\times (0,T),\\
			(\tilde{z}(T),\tilde{z}_\Gamma(T))=(0,0),&\text{ in }\Omega\times \Gamma. 
		\end{cases}
	\end{align*} 
	
	Then, by estimate \eqref{proposition:estimate:L2}, we can assert that 
	\begin{align*}
		&\|(\tilde{z},\tilde{z}_\Gamma)\|_{L^\infty(0,T;\mathbb{L}^2)}^2 + \|(\tilde{z},\tilde{z}_\Gamma)\|_{L^2(0,T;\mathbb{H}^1)}^2\\
		\leq & C\|(\theta g,\theta g_\Gamma)\|_{L^2(0,T;\mathbb{L}^2)}^2 + C\|(\theta' z,\theta ' z_\Gamma)\|_{L^2(0,T;\mathbb{L}^2)}^2.
	\end{align*}
	
	In particular, we can deduce that 
	\begin{align*}
		&\|(z(0),z_\Gamma (0))\|_{\mathbb{L}^2}^2 + \|(z,z_\Gamma)\|_{L^2(0,T/2;\mathbb{H}^1)}^2\\
		\leq & C \|(g,g_\Gamma)\|_{L^2(0,3T/4;\mathbb{L}^2)}^2 + C\|(z,z_\Gamma)\|_{L^2(T/2,3T/4;\mathbb{L}^2)}^2.
	\end{align*}
	
	Since $e^{-2s\hat{\varphi}}\geq C>0$, for each $t\in [0,T/2]$, we have
	\begin{align}
		\label{obs:ineq:02}
		\begin{split}
			&\int_\Omega |z(0)|^2 dx + \int_\Gamma |z_\Gamma (0)|^2 dS+\int_0^{T/2} \int_\Omega e^{-2s\hat{\varphi}} (|z|^2 + |\nabla z|^2) dxdt\\
			&+\int_0^{T/2}\int_\Gamma e^{-2s\hat{\varphi}} (|z_\Gamma|^2 + |\nabla_\Gamma z|^2)dSdt\\
			\leq & C\int_0^{3T/4}\int_\Omega |g|^2 dxdt + C\int_0^{3T/4}\int_\Gamma |g_\Gamma|^2 dSdt+C \int_{T/2}^{3T/4}\int_\Omega |z|^2 dxdt\\
			&+C\int_{T/2}^{3T/4}\int_\Gamma |z_\Gamma|^2 dSdt. 
		\end{split}
	\end{align} 
	In order to estimate the last two terms of \eqref{obs:ineq:02}, we use the fact that 
	\begin{align*}
		0<C\leq e^{-2s\varphi(x,t)}\xi^3,\quad \forall t\in (T/2,3T/4),
	\end{align*}
	to obtain
	\begin{align}
		\nonumber 
		&\int_{T/2}^{3T/4} \int_\Omega |z|^2 dxdt + \int_{T/2}^{3T/4}\int_\Gamma |z_\Gamma|^2 dSdt\\
		\nonumber 
		\leq & \int_{T/2}^{3T/4}\int_\Omega e^{-2s\varphi} \xi^3 |z|^2 dxdt + C\int_{T/2}^{3T/4}\int_\Gamma e^{-2s\varphi} \xi^3 |z_\Gamma|^2 dSdt\\
		\label{obs:ineq:03}
		\leq & C\int_0^T\int_\Omega e^{-2s\varphi}|g|^2 dxdt+C\int_0^T\int_\Gamma e^{-2s\varphi} |g_\Gamma|^2 dSdt+C\int_0^T\int_\omega e^{-2s\varphi}\xi^3 |z|^2 dxdt,
	\end{align} 
	where we have used the Carleman estimate in Theorem \ref{thm:Carleman:estimate} with $s>0$ and $\lambda>0$ fixed. Now, combining \eqref{obs:ineq:02} and \eqref{obs:ineq:03} and using that $-\check \varphi(t)\geq -\varphi(x,t)$ for all $(x,t)\in \overline{\Omega}\times (0,T)$,
	\begin{align}
		\label{obs:ineq:04}
		\begin{split}
			&\int_\Omega |z(0)|^2 dx + \int_\Gamma |z_\Gamma (0)|^2 dS+\int_0^{T/2}\int_\Omega e^{-2s\hat{\varphi}}(\check{\xi}^3|z|^2 + \check{\xi}|\nabla z|^2)dxdt\\
			&+\int_0^{T/2}\int_\Gamma e^{-2s\hat{\varphi}}(\check{\xi}^3|z_\Gamma|^2 + \check{\xi}|\nabla_\Gamma z_\Gamma|^2) dSdt\\
			\leq & C\int_0^T\int_\Omega e^{-2s\check{\varphi}} |g|^2 dxdt + C\int_0^T\int_\Gamma e^{-2s\check{\varphi}} |g_\Gamma|^2 dSdt + C\int_0^{T}\int_\omega e^{-2s\check{\varphi}} \hat{\xi}^3 |z|^2 dxdt. 
		\end{split}
	\end{align} 
	
	On the other hand, by Carleman estimate \eqref{thm:Carleman:ineq} again, the following inequality holds
	\begin{align}
		\label{obs:ineq:05}
		\begin{split}
			&\int_{T/2}^T\int_\Omega e^{-2s\hat{\varphi}} (\check{\xi}^3|z|^2 +\check{\xi} |\nabla z|^2) dxdt+\int_{T/2}^T\int_\Gamma e^{-2s\hat{\varphi}} (\check{\xi}^3|z_\Gamma|^2+ \check{\xi}|\nabla_\Gamma z_\Gamma|^2) dSdt\\
			\leq & C \int_0^T\int_\Omega e^{-2s\check{\varphi}}|g|^2 dxdt + C\int_0^T\int_\Gamma e^{-2s\check{\varphi}} |g_\Gamma|^2 dSdt+ C\int_0^T\int_\omega e^{-2s\check{\varphi}}\hat{\xi}^3|z|^2 dxdt.
		\end{split}
	\end{align}
	
	Finally, adding inequalities \eqref{obs:ineq:04} and \eqref{obs:ineq:05}, we deduce the observability inequality \eqref{observability:inequality}. This ends the proof of the Proposition \ref{proposition:observability:inequality}.  
\end{proof}

\subsection{Proof of the null controllability for the linear system}
From the observability inequality \eqref{observability:inequality}, we can deduce  the null controllability of the linear system \eqref{linear:system:y}. In the following, for $r\in [1,+\infty]$, a linear space $\mathcal{H}$ 
and a measurable function $\rho:(0,T) \longrightarrow \R$, 
we shall consider the notation
\begin{align*}
	L^r(\rho (0,T);\mathcal{H})=\{y\in L^r(0,T;\mathcal{H}); \rho y\in L^r(0,T;\mathcal{H})\}.
\end{align*} 

Consider the Banach space $\mathcal{V}$ defined by
\begin{align}
	\label{def:V}
	\begin{split}
	\mathcal{V}:=\{
	(y,y_\Gamma,h) :& (y,y_\Gamma)\in L^2(e^{s\check{\varphi}}(0,T);\mathbb{L}^2),\, h \mathbbm{1}_\omega \in L^2(e^{s\check{\varphi}}\hat{\xi}^{-3/2}(0,T); L^2(\Omega)), \\
	 & (L(y)-\mathbbm{1}_\omega h,L_\Gamma(y,y_\Gamma))\in L^2(e^{s\hat{\varphi}}\check{\xi}^{-3/2} (0,T);\mathbb{L}^2),\\
	 & (y,y_\Gamma)\in L^2(e^{\frac{1}{3}s\hat{\varphi}} (0,T); \mathbb{H}^2) \cap L^\infty(e^{\frac{1}{3}s\hat{\varphi}}(0,T);\mathbb{H}^1) 
	\},
	\end{split}
\end{align}
endowed by its natural norm:
\begin{align*}
	\|(y,y_\Gamma,h)\|_\mathcal{V}^2 :=& \|e^{s\check{\varphi}}(y,y_\Gamma)\|_{L^2(0,T;\mathbb{L}^2)}^2 + \|e^{s\check{\varphi}} \hat{\xi}^{3}h\mathbbm{1}_\omega\|_{L^2(\Omega\times (0,T))}^2\\
	& +\|e^{s\hat{\varphi}}\check{\xi}^{-3/2} (L(y)-\mathbbm{1}_\omega h,N(y,y_\Gamma))\|_{L^2(0,T; \mathbb{L}^2)}^2\\
	&+\|e^{\frac{1}{3}s\hat{\varphi}}(y,y_\Gamma)\|_{L^2(0,T;\mathbb{H}^2)}^2+ \|e^{\frac{1}{3}s\hat{\varphi}} (y,y_\Gamma)\|_{L^\infty(0,T;\mathbb{H}^1)} ^2.   
\end{align*}
\begin{proposition}
	\label{proposition:null:controllability}
	Let $(y_0,y_{\Gamma,0})\in \mathbb{L}^2$ and assume that 
	\begin{align}
		\label{assump:f:fGamma}
		(f,f_\Gamma)\in L^2(e^{s\hat{\varphi}}\check{\xi}^{-3/2} (0,T);\mathbb{L}^2).
	\end{align}
	Then, we can find a control $h$ such that the associated solution $(y,y_\Gamma)$ of \eqref{linear:system:y} satisfies $(y,y_\Gamma,h)\in \mathcal{V}$. In particular, we have
	\begin{align*}
		y(\cdot, T)=0,\text{ in }\Omega,\quad y_\Gamma (\cdot, T)=0,\text{ on }\Gamma.
	\end{align*} 
\end{proposition}
\begin{proof}
	For our purposes, we define 
	\begin{align*}
		P_0:=\left\{(z,z_\Gamma)\in C^\infty(\overline{\Omega}\times [0,T]) \times C^\infty(\Gamma\times [0,T])\,:\, z\big|_\Gamma =z_\Gamma\text{ on }\Gamma\times [0,T] \right\}.
	\end{align*}
	
	Let $\mathbbm{a}:P_0\times P_0 \to \mathbb{R}$ be the bilinear form 
	\begin{align*}
		\mathbbm{a}((z,z_\Gamma),(w,w_\Gamma)):=&\Re\int_0^T\int_\Omega e^{-2s\check{\varphi}} L^* z \overline{L^* (w)} dxdt+\Re\int_0^T\int_\Gamma e^{-2s\check{\varphi}} L_\Gamma^*(z,z_\Gamma)\overline{L_\Gamma^*(w,w_\Gamma)} dSdt\\
		&+\Re\int_0^T\int_\omega e^{-2s\check{\varphi}}\hat{\xi}^3 z\overline{w} dxdt,\quad \forall (z,z_\Gamma),(w,w_\Gamma)\in P_0\times P_0.
	\end{align*}
	
	We also define the linear form $\ell: P_0 \to \mathbb{R}$ as
	\begin{align*}
		\ell (w,w_\Gamma):=&\Re\int_\Omega y_0 \overline{w(0)} dx + \Re\int_\Gamma y_{\Gamma,0}\overline{w_\Gamma (0)} dS+\Re\int_0^T\int_\Omega f\overline{w} dxdt\\
		&+\Re\int_0^T\int_\Gamma f_\Gamma \overline{w_\Gamma} dSdt,\quad \forall (w,w_\Gamma)\in P_0.
	\end{align*}
	
	In view of the observability inequality \eqref{observability:inequality}, it is clear that 
	\begin{align*}
		\|(z,z_\Gamma)\|_{\mathbbm{a}}:=\sqrt{\mathbbm{a}((z,z_\Gamma),(z,z_\Gamma))},\quad \forall (z,z_\Gamma)\in P_0,
	\end{align*}
	defines a norm in $P_0$. Let $P$ be the completion of $(P_0,\|\cdot\|_\mathbbm{a})$. Then, $\mathbbm{a}(\cdot,\cdot)$ is well-defined, continuous and again definite positive on $P$. On the other hand, by  \eqref{assump:f:fGamma} and the
	observability inequality  \eqref{observability:inequality},   it is easy to see that
	\begin{align*}
		|\ell (w,w_\Gamma)| \leq C\|(w,w_\Gamma)\|_\mathbbm{a},\quad \forall (w,w_\Gamma)\in P_0,
	\end{align*} 
	i.e. $\ell$ is continuous in $P_0$ and,  by density,  in $P$. Thus, by Lax-Milgram Theorem,  the variational problem 
	\begin{align*}
		\begin{cases}
			\text{Find }(z,z_\Gamma)\in P\text{ such that }\\
			\mathbbm{a}((z,z_\Gamma),(w,w_\Gamma))=\ell(w,w_\Gamma),\quad \forall (w,w_\Gamma)\in P,
		\end{cases}
	\end{align*}
	possesses exactly one solution $(z^*,z^*_\Gamma)$. Let $(y^*,y^*_\Gamma,h^*)$ defined by
	\begin{align*}
		\begin{cases}
		y^*:=e^{-2s\check{\varphi}}L^*(z^*),&\text{ in }\Omega\times (0,T),\\
			y^*_\Gamma:= e^{-2s\check{\varphi}}L_\Gamma^*(z^*,z_\Gamma^*),&\text{ on }\Gamma\times (0,T),\\
			h^*:=e^{-2s\check{\varphi}}\hat{\xi}^3 z^*,&\text{ in }\omega\times (0,T).
		\end{cases}
	\end{align*}
	
	We point out that $(y^*,y^*_\Gamma,h^*)$ satisfies
	\begin{align*}
		&\int_0^T\int_\Omega e^{2s\check{\varphi}} |y^*|^2 dxdt+\int_0^T\int_\Gamma e^{2s\check{\varphi}} |y_\Gamma^*|^2 dSdt + \int_0^T\int_\omega e^{2s\check{\varphi}}\hat{\xi}^{-3} |h^*|^2 dxdt\\
		=&\mathbbm{a}((z^*,z_\Gamma^*),(z^*,z_\Gamma^*))<\infty,
	\end{align*}
	and also that 
	\begin{align}
		\label{distri:solution:y:star}
		\begin{split}
			&\Re\int_0^T\int_\Omega y^* \overline{L^*(w)}dxdt+\Re\int_0^T\int_\Gamma y_\Gamma^* \overline{L_\Gamma^*(w,w_\Gamma)} dSdt+\Re \int_0^T\int_\omega h^* \overline{w} dxdt\\
			=&\Re \int_\Omega y_0 \overline{w(0)} dx +\Re \int_\Gamma y_{\Gamma,0}\overline{w_\Gamma (0)} dS+\Re \int_0^T\int_\Omega f\overline{w}dxdt\\
			&+\Re\int_0^T\int_\Gamma f_\Gamma \overline{w_\Gamma} dSdt, 
		\end{split}
	\end{align}
	i.e., from \eqref{distri:solution:y:star} we deduce that $(y^*,y_\Gamma^*)\in C^0([0,T];\mathbb{L}^2)\cap L^2(0,T;\mathbb{H}^1)$  is a distributional solution with control $h^*$ and initial datum $(y_0,y_{\Gamma,0})$ such that $y^*(\cdot,T)=0$ in $\Omega$ and $y_\Gamma (\cdot,T)=0$ on $\Gamma$. 
	
	It remains to check that
	\begin{align}
	\label{regularity:ystar} 
	(y^*,y_\Gamma^*)\in L^2(e^{\frac{1}{3}s\hat{\varphi}}(0,T);\mathbb{H}^2) \cap L^\infty(e^{\frac{1}{3}s\hat{\varphi}}(0,T);\mathbb{H}^1).
	\end{align}
	In order to do that, we notice that the new variables 
	\begin{align*}
		(y^\star, y_\Gamma ^\star):=(e^{\frac{1}{3}s\hat{\varphi}} y^*,e^{\frac{1}{3}s\hat{\varphi}} y_\Gamma^*)
	\end{align*}
	solves the problem 
	\begin{align}
		\label{problem:ystar}
		\begin{cases}
			L(y^\star)=e^{\frac{1}{3}s\hat{\varphi}}(f+h\mathbbm{1}_\omega) + (e^{\frac{1}{3}s\hat{\varphi}})_t y^*,&\text{ in }\Omega\times (0,T),\\
		L_\Gamma(y^\star,y_\Gamma^\star)=e^{\frac{1}{3}s\hat{\varphi}} f_\Gamma +(e^{\frac{1}{3}s\hat{\varphi}})_t y_\Gamma^*,&\text{ in }\Gamma\times (0,T),\\
			y^\star =y_\Gamma^\star,&\text{ on }\Gamma \times (0,T),\\
			(y(0),y_\Gamma(0))=e^{\frac{1}{3}s\hat{\varphi}(0)}(y_0,y_{\Gamma,0}),&\text{ in }\Omega\times \Gamma. 
		\end{cases}
	\end{align}
	
	Since 
	\begin{align*}
		\left|\left(e^{\frac{1}{3}s\hat{\varphi}} \right)_t y^* \right|\leq C\hat{\xi}^2 e^{\frac{1}{3}s\hat{\varphi}} |y^*| \leq C e^{s\check{\varphi}}|y^\star|\in L^2(0,T;\mathbb{L}^2), 
	\end{align*}
	and
	\begin{align*}
		\left(e^{\frac{1}{3}s\hat{\varphi}}(f+h\mathbbm{1}_\omega),e^{\frac{1}{3}s\hat{\varphi}}f_\Gamma \right)\in L^2(0,T;\mathbb{L}^2), 
	\end{align*}
	and since $e^{\frac{1}{3}s\hat{\varphi}(0)}(y_0,y_{\Gamma,0})\in \mathbb{H}^1$, it follows from Proposition \ref{prop:estimate:H2} that the solution $(y^\star,y_\Gamma^\star)$ of \eqref{problem:ystar} satisfies 
	\begin{align*}
		(y^\star,y_\Gamma^\star)\in L^2(0,T;\mathbb{H}^2 \cap C^0([0,T];\mathbb{H}^1)),
	\end{align*}
	which is equivalently to \eqref{regularity:ystar}. We conclude that $(y^*,y_\Gamma^*,h^*)\in \mathcal{V}$. This ends the proof of the Proposition \ref{proposition:null:controllability}.  
\end{proof}
%

\section{Proof of the Theorem \ref{main:thm}}
\label{section:proof:main:result}
  
  To prove the main theorem of this article, we shall use a local inversion argument. This will be done by using the following local inversion mapping theorem in Banach spaces (see for instance \cite{Alekseev87}, page 107).
  
  \begin{theorem}
  	\label{thm:inversion}
  	Let $\mathcal{B}_1$ and $\mathcal{B}_2$ be two Banach spaces and let $\mathcal{A}:\mathcal{B}_1 \to \mathcal{B}_2$ be a $C^1(\mathcal{B}_1;\mathcal{B}_2)$ function. Suppose that $b_1\in \mathcal{B}_1$, $b_2=\mathcal{A}(b_1)$ and that $\mathcal{A}'(b_1): \mathcal{B}_1\to \mathcal{B}_2$ is surjective. Then, there exists $\delta>0$ such that, for every $b'\in \mathcal{B}_2$ satisfying $\|b'-b_2\|_{\mathcal{B}_2}<\delta$, there exists a solution of the equation
  	\begin{align*}
  	\mathcal{A}(b)=b',\quad b\in \mathcal{B}_1.
  	\end{align*}
  	
  	Moreover, there exists a constant $C>0$ such that 
  	\begin{align*}
  		\|b_1-b\|_{\mathcal{B}_1} \leq C \|b_2-b'\|_{\mathcal{B}_2}.
  	\end{align*}
  \end{theorem}
  
  Now, we have all the ingredients to prove the Theorem \ref{main:thm}.
  \begin{proof}[Proof of the Theorem \ref{main:thm}]
  	Consider the spaces
  	\begin{align*}
  	\mathcal{B}_1:=\mathcal{V},\quad \mathcal{B}_2:=L^2(e^{s\hat{\varphi}}\check{\xi}^{-3/2}(0,T);\mathbb{L}^2)\times \mathbb{L}^2,
  	\end{align*}
  	where $\mathcal{V}$ is the linear space defined in \eqref{def:V}.
  	 
  	Consider $\mathcal{A}:\mathcal{B}_1\to \mathcal{B}_2$ be the operator defined by
  	\begin{align*}
  	\mathcal{A}(u,u_\Gamma,h)=&\left(Lu+c(1+\gamma i)|u|^2u-\mathbbm{1}_\omega h,
  	L_\Gamma(u,u_\Gamma) + c(1+\gamma i)|u_\Gamma|^2 u_\Gamma, u(0),u_\Gamma(0) \right).
  	\end{align*}
  	We choose, $b_1=(0,0,0)$ and $b_2=\mathcal{A}(0,0,0)=(0,0,u_0,u_{\Gamma,0})$. In order to use the Theorem \ref{thm:inversion}, we shall check that the following two assertions:
  	\begin{itemize}
  		\item[(a)] $\mathcal{A}'(0,0,0):\mathcal{B}_1 \to \mathcal{B}_2$ is surjective.
  		\item[(b)] $\mathcal{A}$ is an operator of class $C^1$ from $\mathcal{B}_1$ to $\mathcal{B}_2$.
  	\end{itemize}
  	
  	A simple computation shows that for all $(u^*,u_\Gamma^*,h^*)\in \mathcal{B}_1$,
  	\begin{align}
  		\label{comp:A:prime}
  		\begin{split}
  		\mathcal{A}'(u,u_\Gamma,h)(u^*,u^*_\Gamma,h^*)=&( L(u^*)+3c(1+\gamma i)|u|^2 u^*-\mathbbm{1}_\omega h^*,L_\Gamma(u^*,u_\Gamma^*)\\
  		&
		+3c(1+\gamma i)|u_\Gamma|^2 u_\Gamma^*,u^*(0),u_\Gamma^*(0)).
  		\end{split}
  	\end{align}
  	
  	In particular, taking $(u,u_\Gamma,h)=(0,0,0)$ in \eqref{comp:A:prime}, we have
  	\begin{align*}
  	\begin{split}
  		\mathcal{A}'(0,0,0)(u^*,u_\Gamma^*,h^*)=&\left( L(u^*)-\mathbbm{1}_\omega h^*,N(u^*,u_\Gamma^*), u^*(0),u_\Gamma^*(0)\right).  
  	\end{split}
  	\end{align*}
  	
  	Now, it is clear that, in view of the Proposition \ref{proposition:null:controllability}, the operator $\mathcal{A}'(0,0,0)$ is surjective. 
  	
  	On the other hand, to prove that $\mathcal{A}\in C^1(\mathcal{B}_1,\mathcal{B}_2)$, it is sufficient to check that the map 
  	$$(u_1,u_{\Gamma,1},h_1),(u_2,u_{\Gamma,2},h_2),(u_3,u_{\Gamma,3},h_3)\mapsto (u_1u_2u_3,u_{\Gamma,1}u_{\Gamma,2}u_{\Gamma,3}),$$
  	is continuous from $\mathcal{V}\times \mathcal{V}\times \mathcal{V}$ to $L^2(e^{s\check{\varphi}}(0,T); \mathbb{L}^2)$. Then, according to H\"older inequality and using that $ L^2(0,T;\mathbb{H}^2) \cap L^\infty(0,T;\mathbb{H}^1) \hookrightarrow  L^6(0,T;\mathbb{L}^9) \hookrightarrow L^6(0,T;\mathbb{L}^6) $, we have 
  	\begin{align}
  	\label{key:step}
  	\begin{split}
  		&\|e^{s\hat{\varphi}} \check{\xi}^{-3/2}(u_1u_2u_3,u_{\Gamma,1}u_{\Gamma,2}u_{\Gamma,3})\|_{L^2(0,T;\mathbb{L}^2)}\\
  		\leq & C\prod_{j=1}^3 \|e^{\frac{1}{3}s\hat{\varphi}} (u_j,u_{\Gamma,j})\|_{L^6(0,T;\mathbb{L}^6)}\\
  		\leq & C\prod_{j=1}^3 \| e^{\frac{1}{3}s\hat{\varphi}} (u_j,u_{\Gamma,j}) \|_{L^2(0,T;\mathbb{H}^2)\cap L^\infty(0,T;\mathbb{H}^1)}\\
  		\leq & C\prod_{j=1}^3 \|(u_j,u_{\Gamma,j},h_j)\|_{\mathcal{V}}.
  		\end{split}		
  	\end{align}
  	
  	Therefore, we have that $\mathcal{A}\in C^1(\mathcal{B}_1,\mathcal{B}_2)$, and we can apply Theorem \ref{thm:inversion} to guarantee the existence of $\delta>0$ such that $\|(u_0,u_{\Gamma,0})\|_{\mathbb{H}^1}\leq \delta$, one can find $(y,y_\Gamma,h)\in \mathcal{B}_1:=\mathcal{V}$ such that the associated solution $(u,u_\Gamma)$ (with initial datum $(u_0,u_{\Gamma,0})$) satisfies 
  	\begin{align*}
  		y(\cdot,T)=0,\text{ in }\Omega,\quad y_\Gamma (\cdot,T)=0,\text{ on }\Gamma.
  	\end{align*}
  	
  	This, ends the proof of the Theorem \ref{main:thm}.
  \end{proof}

	\section*{Acknowledgments}
	
	Nicol\'as Carre\~no has been funded by ANID FONDECYT 1211292. Alberto Mercado has been funded by ANID FONDECYT 1211292 and ANID Millennium Science Initiative Program, Code NCN19-161. Roberto Morales has been funded by ANID FONDECYT 3200830.
	     
  
  \bibliographystyle{plain}

\end{document}